\documentclass[a4paper,10pt]{article}

\usepackage[latin1]{inputenc}
\usepackage{amsmath}
\usepackage{amsfonts}
\usepackage{amssymb}
\usepackage{amsthm}
\usepackage[english]{babel}
\allowdisplaybreaks[2]
\usepackage{paralist}
\usepackage[colorlinks=true]{hyperref}
\hypersetup{urlcolor=blue, citecolor=red}
\usepackage{hyperref}

\newtheorem{theorem}{Theorem}[section]
\newtheorem{corollary}[theorem]{Corollary}

\newtheorem{lemma}[theorem]{Lemma}
\newtheorem{proposition}[theorem]{Proposition}

\theoremstyle{definition}
\newtheorem{definition}[theorem]{Definition}
\newtheorem{remark}[theorem]{Remark}

\newtheorem{example}[theorem]{Example}
\newtheorem{assumption}[theorem]{Assumption}

\newcommand{\C}{\mathbb{C}}
\newcommand{\R}{\mathbb{R}}
\newcommand{\diag}{\operatorname{diag}\,}

\renewcommand{\H}{\mathcal{H}}

\newcommand{\N}{\mathbb{N}}
\newcommand{\z}{\zeta}
\renewcommand{\sp}[2]{\langle {#1}, {#2} \rangle}

\newcommand{\dom}{D}
\newcommand{\norm}[1]{\left\Vert{#1}\right\Vert}
\newcommand{\abs}[1]{\left|{#1}\right|}
\newcommand{\A}{\mathcal{A}}
\newcommand{\B}{\mathcal{B}}
\newcommand{\T}{\mathcal{T}}
\newcommand{\ran}{\operatorname{ran}\,}

\renewcommand{\Re}{\operatorname{Re}\,}
\renewcommand{\Im}{\operatorname{Im}\,}

\begin{document}

\title{Stability and Stabilization of Infinite-dimensional Linear Port-Hamiltonian Systems}
\author{Bj\"orn Augner\footnote{Fachbereich C -- Mathematik und Naturwissenschaften, Bergische Universit\"at Wuppertal, Gau\ss{}stra\ss{}e 20, D-42119 Wuppertal, Germany. The first author is supported by Deutsche Forschungsgemeinschaft (Grant JA 735/8-1). augner@uni-wuppertal.de} \and Birgit Jacob\footnote{Fachbereich C -- Mathematik und Naturwissenschaften, Bergische Universit\"at Wuppertal, Gau\ss{}stra\ss{}e 20, D-42119 Wuppertal, Germany. jacob@math.uni-wuppertal.de}}

\maketitle

\begin{abstract}
Stability and stabilization of linear port-Hamiltonian systems on infinite-dimensional spaces are investigated.
This class is general enough to include models of beams and waves as well as transport and Schr\"odinger equations with boundary control and observation. The analysis is based on the frequency domain method which gives new results for second order port-Hamiltonian systems and hybrid systems. Stabilizing controllers with colocated input and output are designed.
The obtained results are applied to the Euler-Bernoulli beam.
\end{abstract}



\section{Introduction}

In recent years there has been a growing interest in the stability and stabilization of wave and beam equations.
For several of these equations results for structural damping or boundary feedback have been detected using Lyapunov methods, a Riesz basis approach or frequency domain methods.
A large class of these equations may be written in the form of \emph{port-Hamiltonian systems}
	\begin{equation}
	\frac{\partial x}{\partial t}(t,\z)
	 = \sum_{k=0}^N P_k \frac{\partial^k (\H x)}{\partial \z^k}(t,\z),
	 \quad
	 t \geq 0, \ \z \in (0,1)
	 \label{eqn:1}
	\end{equation}
with suitable boundary conditions.
This class covers in particular the wave equation, the transport equation, the Timoshenko beam equation (all $N = 1$), but also the Schr\"odinger equation and the Euler-Bernoulli beam equation (both $N = 2$).
For distributed parameter systems as port-Hamiltonian systems see \cite{VanDerSchaftMaschke_2002} and in particular the Ph.D thesis \cite{Villegas_2007}.
We follow this unified approach and employ the rich theory of one-parameter $C_0$-semigroups of linear operators (e.g. \cite{EngelNagel_2000}) and, more specifically, some of the stability theory (\cite{ArendtBatty_1988}, \cite{Eisner_2010}, \cite{Gearhart_1978}, \cite{LyubichPhong_1988}, \cite{Pruss_1984}, \cite{VanNeerven_1996}).
Our investigation has the following two parts: stability (or stabilization by static feedback, i.e. pure infinite-dimensional systems) and stabilization by dynamical feedback (i.e. hybrid systems).
We concentrate only on boundary feedback stabilization, although most of our results naturally extend to situations with structural damping.
For the pure infinite-dimensional part already some results for port-Hamiltonian systems have been known, especially for the case $N = 1$ (\cite{Engel_2013}, \cite{JacobZwart_2012}, \cite{VillegasEtAl_2009}) whereas for the case $N = 2$ most of the research has been focussed on particular examples of beam equations (\cite{ChenEtAl_1987}, \cite{ChenEtAl_1987a}, \cite{GuoWangYung_2005}).
On the other hand, for beam equations hybrid systems have been investigated for some time now (\cite{GuoHuang_2004}, \cite{LittmanMarkus_1988}, \cite{LiuLiu_2000}) and recently for SIP controllers with colocated input- and output map a nice result for the case $N = 1$ has been established (\cite{RamirezZwartLeGorrec_2013}).
The latter turns out to be a special case of the results presented here.

This article is organised as follows.
Section \ref{sec:infinite_dimensional_phs} is devoted to pure infinite-di\-men\-sional port-Hamiltonian systems, where in Subsection \ref{sec:semigroup_generation} we derive the contraction semigroup generation theorem for the operator $A$ associated to the evolution equation (\ref{eqn:1}).
However, our main objective is to investigate the asymptotic behaviour of port-Hamiltonian systems.
We focus on two types of stability concepts.
Namely let $(T(t))_{t\geq0}$ be any $C_0$-semigroup on $X$.
We say that $(T(t))_{t\geq0}$ is \emph{asymptotically} (strongly) \emph{stable} if
	\begin{equation}
	 T(t) x
	  \xrightarrow{t \rightarrow \infty} 0,
	  \quad
	  \text{for all} \ x \in X
	\end{equation}
respectively (uniformly) \emph{exponentially stable} if there exist $M \geq 1$ and $\omega < 0$ with
	\begin{equation}
	 \norm{T(t)}
	  \leq M e^{\omega t},
	  \quad
	  t \geq 0.
	\end{equation}
Here $(T(t))_{t\geq0}$ is the $C_0$-semigroup generated by the port-Hamiltonian operator $A$.
Our approach is based on Stability Theorems \ref{thm:Arendt-Batty} and \ref{thm:gearhart}.
These results motivate to introduce properties ASP, AIEP and ESP in Subsection \ref{subsec:properties}.
We then only has to test whether a particular function $f: \dom(A_0) \rightarrow \R_+$ has one of these properties to obtain the corresponding stability result.
The main advantage of using these properties does not lie in the pure infinite-dimensional case (with static feedback), but in the case of dynamical feedback via (finite-dimensional) controllers which we consider later in Section \ref{sec:interconnection}.
In the latter case we use the same properties ASP, AIEP and ESP in order to deduce results for interconnected systems without having to reprove the same auxiliary results once again.
We start with asymptotic (strong) stability and based on the Stability Theorem \ref{thm:Arendt-Batty} by Arendt, Batty, Lyubich and Phong give a general asymptotic stability result for port-Hamiltonian systems.
Then we continue with exponential stability for the case $N = 1$ in Subsection \ref{subsec:exponential_stability_first_order}.
This class of systems has been extensively studied in the book \cite{JacobZwart_2012}.
Originally in \cite{VillegasEtAl_2009} the authors presented an exponential stability result based on some \emph{sideways energy} estimate (Lemma III.1 in \cite{VillegasEtAl_2009}) which goes back to an idea of Cox and Zuazua (Theorem 10.1 in \cite{CoxZuazua_1995}).
We establish the same result using a frequency domain method based on Gearthart's Theorem \ref{thm:gearhart}.
It turns out that by this technique we do not only obtain a different proof for exponential stability of first order port-Hamiltonian systems, but the method extends to a proof for second order systems as well, whereas the idea in \cite{VillegasEtAl_2009} seems to be restricted to the transport equation-like situation for first order systems.
We even present a general exponential stability result for second order port-Hamiltonian systems in Subsection \ref{subsec:exponential_stability_second_order}.
Moreover we give a sufficient condition for second order systems with some special structure which applies in particular to Euler-Bernoulli beam equations.

Section \ref{sec:interconnection} then constitutes a breach since we leave the pure infinite-dimensional setup and consider \emph{hybrid systems} which consist of both a infinite-dimensional subsystem (governed by a port-Hamiltonian partial differential equation) and a finite-dimensional subsystem which we think of as a controller (modelled by an ordinary differential equation).
In applications these situations are characterized by an energy functional which splits into a continuous part and a discrete part.
We interpret the total system as an interconnection of two subsystems which interact with each other by means of boundary control and observation.
We then depict how the theory for the pure infinite-dimensional case naturally carries over to these hybrid systems.
After stating the generation result in Subsection \ref{subsec:interconnection_generation} we obtain a stability result for hybrid systems in Subsection \ref{subsec:interconnection_stability} without additional structure conditions.
For the special class of strictly input passive (SIP) controllers with colocated input and output we then obtain in Subsection \ref{subsec:sip_controllers} a stability result which is much more suitable for applications.
As a special case we rediscover the main result of \cite{RamirezZwartLeGorrec_2013} (which has been proved using a Lyapunov method with the same sideways energy estimate mentioned above).

Finally, in Section \ref{sec:examples} we illustrate how our theoretical results can be used to reobtain some stability results on the Euler-Bernoulli beam equation, namely the situations considered in \cite{ChenEtAl_1987a} and \cite{GuoHuang_2004}.
In the latter case we encounter a situation where the finite-dimensional controller naturally appears in the modelling of the problem.

\section{Infinite-dimensional Port-Hamiltonian Systems}
\label{sec:infinite_dimensional_phs}

Throughout this paper we use the following notations.
For any Hilbert space $X$ we denote by $\sp{\cdot}{\cdot}$ its inner product (which is linear in the \emph{second} component).
Moreover $\B(X,Y)$ denotes the space of linear and bounded operators $X \rightarrow Y$ where as usual $\B(X) := \B(X,X)$.
For any closed linear operator $A: \dom(A) \subset X \rightarrow X$ we have the resolvent set $\rho(A)$, the spectrum $\sigma(A)$ and write $R(\lambda,A) := (\lambda I - A)^{-1}$ for the resolvent operator and $\sigma_p(A)$ for the point spectrum of $A$.
We investigate port-Hamiltonian systems of order $N \in \N$, given by the partial differential equation
	\begin{equation}
	\frac{\partial x}{\partial t}(t,\z)
	 = \sum_{k=0}^N P_k \frac{\partial^k (\H x)}{\partial \z^k}(t,\z),
	 \quad
	 t \geq 0, \ \z \in (0,1).
	 \label{gencontrsgr-eqn-1}
	\end{equation}
Here $P_k \in \C^{d \times d}$, $k = 0, 1, \ldots, N$, always denotes some complex matrices satisfying the condition
	\begin{equation}
	P_k^*
	 = (-1)^{k-1} P_k,
	 \qquad
	 k \geq 1.
	\end{equation}
(Note that we do not require $P_0$ to be skew-adjoint.)
Moreover we always assume that $P_N$ is invertible.
The \emph{Hamiltonian density matrix function} $\H: (0,1) \rightarrow \C^{d \times d}$ is a measurable function such that there exist $0 < m \leq M$ such that for almost every $\z \in (0,1)$ the matrix $\H(\z)$ is self-adjoint and
	\begin{equation}
	 m \abs{\xi}^2
	  \leq \xi^* \H(\z) \xi
	  \leq M \abs{\xi}^2,
	  \quad
	  \xi \in \C^d.
	\end{equation}
We then say that $\H$ is \emph{uniformly positive}.
In this paper we consider the \emph{energy state space} $X = L_2(0,1;\C^d)$ with the inner product
	\begin{equation}
	 \sp{f}{g}_{\H}
	  := \int_0^1 f^*(\z) \H(\z) g(\z) d\z,
	  \quad f, g \in X. 
	\end{equation}
Note that $\norm{\cdot}_{\H}$ is equivalent to the standard $L_2$-norm $\norm{\cdot}_{L_2}$.

The operator $A_0: \dom(A_0) \subset X \rightarrow X$ corresponding to equation (\ref{gencontrsgr-eqn-1}) is given by
	\begin{align}
	A_0 x
	 &= \sum_{k=0}^N P_k \frac{d^k}{d \z^k} (\H x),
	 \nonumber \\
	\dom(A_0)
	 &= \{x \in X: \H x \in H^{N}(0,1;\C^d) \}.
	 \label{gencontrsgr-eqn-2}
	\end{align}
Thanks to the invertibility of $P_N$ the operator $A_0$ is closed.
\begin{lemma}
\label{lem:equivalent_norms}
The operator $A_0$ is a closed operator
and its graph norm is equivalent to the norm $\norm{\H \cdot}_{H^N}$.
\end{lemma}
Let
	\begin{equation}
	 \Phi: H^N(0,1;\C^d) \rightarrow \C^{2Nd}, \ \Phi(x) = (x(1), \ldots, x^{(N-1)}(1), x(0) \ldots, x^{(N-1)}(0))
	 \nonumber
	\end{equation}
be the \emph{boundary trace operator} and introduce the \emph{boundary port variables} $\left( \begin{smallmatrix} f_{\partial, \H x} \\ e_{\partial, \H x} \end{smallmatrix} \right)$ defined via
	\begin{align}
	\left(  \begin{array}{c} f_{\partial, \H x} \\ e_{\partial, \H x} \end{array} \right)
	 &= \frac{1}{\sqrt{2}} \left( \begin{array}{cc} Q & -Q \\ I & I  \end{array} \right) \Phi(\H x)
	 \nonumber \\
	 \nonumber \\
	Q_{ij}
	 &= \left\{ \begin{array}{ll} (-1)^{j-1} P_{i+j-1}, & i+j \leq N + 1 \\ 0, & \text{else}. \end{array} \right.
	 \label{gencontrsgr-eqn-4}
	\end{align}
Note that the boundary port variables do not depend on the matrix $P_0$.
If $P_0 = - P_0^*$ is skew-adjoint, the boundary port variables determine $\Re \sp{A_0 x}{x}$.

\begin{lemma}
\label{gencontrsgr-lem-4}
Assume $P_0^* = -P_0$.
Then the operator $A_0$ satisfies
	\begin{equation}
	2 \Re \sp{A_0 x}{x}_{\H}
	 = \Re \sp{f_{\partial, \H x}}{e_{\partial, \H x}}_{\C^{2Nd}},
	 \qquad
	 x \in \dom(A_0).
	 \label{gencontrsgr-eqn-7}
	\end{equation}
\end{lemma}

\subsection{Generation of Contraction Semigroups}
\label{sec:semigroup_generation}

Since we did not impose any boundary conditions in equation (\ref{gencontrsgr-eqn-2}), we could not expect $A_0$ to generate a $C_0$-semigroup
(in fact, $\sigma_p(A_0) = \C$).
However, for suitable boundary conditions, defining a subspace $\dom(A) \subset \dom(A_0)$ the restricted operator $A = A_0|_{D(A)}$ has the generator property.
For this purpose, let $W \in \C^{Nd \times 2Nd}$ be a full rank matrix and define the operator $A$ by
	\begin{align}
	A
	 &= A_0|_{D(A)},
	 \nonumber \\
	\dom(A)
	 &= \{x \in D(A_0): W \left( \begin{array}{c} f_{\partial, \H x} \\ e_{\partial, \H x} \end{array} \right) = 0\}.
	 \label{gencontrsgr-eqn-3}
	\end{align}

Note that thanks to the invertibility of $P_N$, the matrix $\left( \begin{smallmatrix} Q & -Q \\ I & I  \end{smallmatrix} \right)$ is invertible (see Lemma 3.4 in \cite{LeGorrecZwartMaschke_2005}) and thus the condition $W \left( \begin{smallmatrix} f_{\partial, \H x} \\ e_{\partial, \H x} \end{smallmatrix} \right) = 0$ may be equivalently expressed as $W' \Phi(\H x) = 0$ for a suitable matrix $W'$.

Using the Lumer-Phillips Theorem II.3.15 in \cite{EngelNagel_2000} the generators of contraction semigroups have been characterized by a simple matrix condition or alternatively by dissipativity of the operator.
Note that usually the hard part of proving that an operator $A$ generates a contraction semigroup is the range condition $\ran (\lambda I - A) = X$ for some $\lambda > 0$.

\begin{theorem}
\label{gencontrsgr-thm-7}
The following are equivalent.
	\begin{enumerate}
	 \item $A$ generates a contraction $C_0$-semigroup,
	 \item $A$ is dissipative, i.e. $\Re \sp{Ax}{x}_{\H} \leq 0$, \quad for all $x \in D(A)$,
	 \item $W \Sigma W^* \geq 0$ \ and \ $\Re P_0 \leq 0$
	\end{enumerate}
where $\Sigma = \left( \begin{smallmatrix} 0 & I \\ I & 0 \end{smallmatrix} \right) \in \C^{2d \times 2d}$.
In that case $A$ has compact resolvent.
\end{theorem}

Note that this result is a combination of Theorem 7.2.4 in \cite{JacobZwart_2012} where the authors focus only on the case $N = 1$ and Theorem 4.1 in \cite{LeGorrecZwartMaschke_2005} where the general case of $N$-th order Port-Hamiltonian systems is treated for the equivalence of parts 1.\ and 2.
However in both cases the authors only treat the case $P_0 = - P_0^*$.
For the general case where $P_0^* \not= - P_0$ is not skew-adjoint we use a perturbation argument.

\begin{proof}
Let us first assume that $P_0 = - P_0^*$ is skew-adjoint.
The equivalence of conditions 1.\ and 3.\ is due to Theorem 4.1 in \cite{LeGorrecZwartMaschke_2005}.
The implication 1.\ $\Rightarrow$ 2.\ results from the Lumer-Phillips Theorem II.3.15 in \cite{EngelNagel_2000}.
For the implication 2.\ $\Rightarrow$ 1.\ one only needs to show the range condition $\ran (I - A) = X$ (thanks to the Lumer-Phillips result).
This can be done similar as in the proof of Theorem 7.2.4 in \cite{JacobZwart_2012} (with obvious modifications).
We leave the details to the interested reader.

Let us concentrate on the situation where $P_0 \not= -P_0^*$,
i.e.
	\begin{equation}
	G_0
	 := - \frac{1}{2} ( P_0 + P_0^*) \not= 0
	 \nonumber
	\end{equation}
Of course, the implication $1.\ \Rightarrow 2.$\ follows by the Lumer-Phillips Theorem II.3.15 in \cite{EngelNagel_2000}.
Next we show that 2.\ implies 1.\
Let us write $\tilde A := A + G_0 \H$.
If we can show that $\tilde A$ generates a contractive $C_0$-semigroup, then also $A$ generates a $C_0$-semigroup by the Bounded Perturbation Theorem III.1.3 in \cite{EngelNagel_2000} which then is contractive since its generator is dissipative.
Since $\tilde A$ is a port-Hamiltonian operator with skew-adjoint $\tilde P_0$ it suffices to prove dissipativity of $\tilde A$.
\emph{Assume} $\tilde A$ were not dissipative.
Then by Lemma \ref{lem:auxiliary-dissipativity} below there exists a $X$-null sequence $(x_n)_{n \geq 1}$ in $\dom(A)$ with $\Re \sp{\tilde A x_n}{x_n}_{\H} = 1$
so
	\begin{equation}
	0
	 \geq \Re \sp{Ax_n}{x_n}_{\H}
	 = \Re \sp{\tilde Ax_n}{x_n}_{\H} - \sp{G_0 \H x_n}{x_n}_{\H}
	 \rightarrow 1,
	\end{equation}
which leads to a contradiction.
Hence $\tilde A$ generates a contraction semigroup and so does $A$.
Thus 1.\ and 2.\ are equivalent also in this case.
Further we obtain that if 1.\ or 2.\ holds then $\tilde A$ generates a $C_0$-semigroup, so $W \Sigma W^* \geq 0$.
Moreover for any $\H x \in C_c^{\infty}(0,1;\C^d) \subset \dom(A)$ we obtain
	\begin{equation}
	\Re \sp{Ax}{x}_{\H} = \Re \sp{P_0 \H x}{x}_{\H} = \Re \sp{P_0 \H x}{\H x}_{L_2} \leq 0
	\end{equation}
by 2.\ and hence choosing $\H x = \phi \xi$ for $\phi \in C_c^{\infty}(0,1;\C)$ and $\xi \in \C^d$ it follows $\Re P_0 \leq 0$, so 3.\ holds.
Finally, from 3.\ it follows that $\tilde A$ (as introduced above) generates a contraction semigroup and hence does $A = \tilde A - G_0 \H$ by Theorem III.2.7 in \cite{EngelNagel_2000} and the dissipativity of $-G_0 \H$.
\end{proof}

In the proof we used the following.

\begin{lemma}
\label{lem:auxiliary-dissipativity}
Let $\Re P_0 = 0$.
If $A$ is not dissipative, then there exists a sequence $(x_n)_{n \geq 1}$ in $\dom(A)$ with $\Re \sp{A x_n}{x_n}_{\H} = 1$ and $x_n$ converging to $0$ in $X$.
\end{lemma}
\begin{proof}
Let $x \in \dom(A)$ with $\Re \sp{A x}{x}_{\H} = 1$.
Now for any $n \in \N$ let $y_n \in H^N(0,1;\C^d)$ be such that $\norm{y_n}_{L_{\infty}} \leq 2 \norm{\H x}_{L_{\infty}}$ and
	\begin{equation}
	 y_n(\z)
	  = \left\{ 
	   \begin{array}{ll}
            (\H x)(\z), & \z \in (0,\frac{1}{2n}) \cup (1 - \frac{1}{2n}, 1) \\
            0, & \z \in (\frac{1}{n}, 1 - \frac{1}{n})
           \end{array}
           \right.
	\end{equation}
Then for $x_n := \H^{-1} y_n$ we obtain
	\begin{equation}
	\norm{x_n}_{L_2}^2
	 \leq c \int_{(0,1/n) \cup (1-1/n,1)} \abs{x(\z)}^2 d\z
	 \xrightarrow{n \rightarrow \infty} 0
	\end{equation}
and
	\begin{equation}
	e_{\partial, \H x_n} = e_{\partial, \H x},
	 \quad
	 f_{\partial, \H x_n} = f_{\partial, \H x},
	 \quad
	 n \in \N,
	\end{equation}
so consequently $\Re \sp{\tilde A x_n}{x_n}_{\H} = \Re \sp{\tilde Ax}{x}_{\H} = 1$ for all $n \in \N$.
\end{proof}

\subsection{Sufficient Conditions for Stability}
\label{subsec:properties}

Our main tools to deduce stability results are the following two theorems.

\begin{theorem}[Asymptotic Stability]
\label{thm:Arendt-Batty}
Let $B$ generate a bounded $C_0$-semigroup $(S(t))_{t\geq0}$ on a Banach space $Y$ and assume that $\sigma_r(B) \cap i \R = \emptyset$.
If $\sigma(B) \cap i\R$ is countable, then $(S(t))_{t\geq0}$ is asymptotically stable.
\end{theorem}
Here $\sigma_r(B) := \{ \lambda \in \C: \ran (\lambda I - B) \ \text{not dense in} \ Y\}$ denotes the \emph{residual spectrum} of $B$ which coincides with the point spectrum of the adjoint operator $B'$.
\begin{proof}
See Stability Theorem 2.4 in \cite{ArendtBatty_1988} (or the theorem in \cite{LyubichPhong_1988}).
\end{proof}

Note that in particular for generators $B$ with compact resolvent we have asymptotic stability if and only if $\sigma_p(B) \subset \C_0^-:= \{\lambda \in \C: \Re \lambda < 0\}$.
The second result requires Hilbert space structure.

\begin{theorem}[Exponential Stability]
\label{thm:gearhart}
Let $B$ generate a bounded $C_0$-semigroup $(S(t))_{t\geq0}$ on a Hilbert space $Y$.
Then $(S(t))_{t\geq0}$ is exponentially stable
if and only if
	\begin{equation}
	 \sigma(B)
	  \subseteq \C_0^-,
	 \quad
	 \sup_{\omega \in \R} \norm{R(i\omega,B)} < +\infty.
	 \nonumber
	\end{equation}
\end{theorem}
\begin{proof}
See Theorem 4 and Corollary 5 in \cite{Pruss_1984}.
\end{proof}

\begin{remark}
The uniform boundedness of the resolvent on $i \R$ in Theorem \ref{thm:gearhart} is equivalent to the condition
	\begin{equation}
         \left.
          \begin{array}{l}
           (x_n,\beta_n) \subset \dom(B) \times \R \\
           \sup_{n \in \N} \norm{x_n} < + \infty, \
           \abs{\beta_n} \rightarrow \infty \\
           B x_n - i \beta_n x_n \rightarrow 0
          \end{array}
	 \right\}
         \Longrightarrow \
         x_n \rightarrow 0.
	\end{equation}
\end{remark}

For the moment let $Y$ be any Hilbert space.
The following definition enables us to lift stability results to hybrid systems which we investigate later on.

\begin{definition}
Let a linear operator $B: \dom(B) \subset Y \rightarrow Y$ be given.
We say that a function $f: \dom(B) \rightarrow \R_+$ has the property
	\begin{itemize}
	 \item ASP (for the operator $B$) if for all $\beta \in \R$ and $x \in \dom(B)$
		\begin{equation}
		i \beta x = Bx
		 \ \text{and} \
		f(x) = 0
		 \quad
		 \Rightarrow
	 	 \quad
		x = 0,
		\label{ASP}
		\end{equation}
	\item AIEP (for the operator $B$) if for all sequences $(x_n,\beta_n)_{n \geq 1} \subset \dom(B) \times \R$ with $\sup_{n \in \N} \norm{x_n} < + \infty$ and $\abs{\beta_n} \rightarrow + \infty$
		\begin{equation}
		i \beta_n x_n - B x_n \rightarrow 0
		 \ \text{and} \
		f(x_n) \rightarrow 0
		 \quad
		 \Rightarrow
		 \quad
		x_n \rightarrow 0,
		 \label{AIEP}
		\end{equation}
	\item ESP (for the operator $B$) if it has properties ASP and AIEP.
	\end{itemize}
\end{definition}

Note the following property which easily may be verified using the above definition.

\begin{lemma}
\label{lem:properties}
Let $B \subset B_0$ be linear operators,
i.e.
	\begin{equation}
	\dom(B) \subset \dom(B_0),
	 \quad
	 B_0|_{\dom(B)} = B
	 \nonumber
	\end{equation}
and $f: \dom(B) \rightarrow \R_+$ and $f_0: \dom(B_0) \rightarrow \R_+$ with $\kappa f_0|_{\dom(B_0)} \leq f$ for some $\kappa > 0$.
If $f_0$ has the property ASP or AIEP or ESP (for the operator $B_0$),
then also $f$ has the property ASP or AIEP or ESP (for the operator $B$), respectively.
\end{lemma}

The abbreviations ASP, AIEP and ESP stand for \emph{asymptotic stability property}, \emph{asymptotic implies exponential stability property} and \emph{exponential stability property},
where a typical choice of $f$ are functions of the form
	\begin{equation}
	f(x)
	 = \sum_{k=0}^{N-1} \alpha_{0,k} \abs{(\H x)^{(k)}(0)}^2 + \alpha_{1,k} \abs{(\H x)^{(k)}(1)}^2
	\end{equation}
for some non-negative constants $\alpha_{j,k} \geq 0$.
That the above terminology is indeed appropriate is the statement of the following lemma.

\begin{lemma}
\label{lem:3.2}
Let $B$ have compact resolvent and generate a $C_0$-semigroup $(S(t))_{t\geq0}$ on $Y$ and assume that for some function $f: \dom(B) \rightarrow \R_+$ 
	\begin{equation}
	\Re \sp{Bx}{x} \leq - f(x),
	 \quad
	 x \in \dom(B).
	\end{equation}
Then
	\begin{enumerate}
	\item If $f$ has property ASP then $(S(t))_{t\geq0}$ is asymptotically (strongly) stable.
	\item If $f$ has property AIEP and $\sigma_p(B) \cap i\R = \emptyset$ then $(S(t))_{t\geq0}$ is (uniformly) exponentially stable.
	\item If $f$ has property ESP then $(S(t))_{t\geq0}$ is (uniformly) exponentially stable.
	\end{enumerate}
\end{lemma}

\begin{proof}
1.) If $f$ has property ASP and $i \beta x = B x$ for some $x \in \dom(B)$ and $\beta \in \R$ then
	\begin{equation}
	f(x)
	 \leq - \Re \sp{Bx}{x}
	 = - \Re \sp{i \beta x}{x}
	 = 0
	\end{equation}
and by the property ASP it follows $x = 0$, so $i \R \cap \sigma_p(B) = \emptyset$ and asymptotic stability follows from Stability Theorem \ref{thm:Arendt-Batty}.

2.) Since $\sigma_p(B) \cap i \R = \emptyset$ then $\sigma(B) = \sigma_p(B) \subset \C_0^-$.
Further if $(x_n,\beta_n)_{n\geq1} \subset \dom(B) \times \R$ with $\norm{x_n}_{Y} \leq c$ and $\abs{\beta_n} \rightarrow + \infty$ such that $B x_n - i\beta_n x_n \rightarrow 0$ it follows that
	\begin{equation}
	0
	 \leftarrow \Re \sp{i \beta_n x_n - B x_n}{x_n} 
	 \geq f(x_n)
	 \geq 0,
	\end{equation}
i.e. $f(x_n) \rightarrow 0$ and by property AIEP this leads to $x_n \rightarrow 0$ so exponential stability follows from Stability Theorem \ref{thm:gearhart}.

3.\ is a direct consequence of 1.\ and 2.
\end{proof}

\subsection{Asymptotic Stability of Port-Hamiltonian Systems}
\label{subsec:asymptotic_stability}

An example for a function $f: \dom(A_0) \rightarrow \R_+$ which has property ASP is the square of the Euclidean norm of $\H x(\z)$ and its derivatives at position $\z = 0$.
(Of course, the choice $\z = 1$ is possible as well.)
The asymptotic stability result reads as follows.

\begin{proposition}
\label{prop:asymptotic_stability}
Assume that $A$ satisfies
	\begin{align}
	\Re \sp{Ax}{x}_{\H}
	 &\leq - \kappa \sum_{k=0}^{N-1} \abs{(\H x)^{(k)}(0)}^2,
	 \quad
	 x \in \dom(A),
	 \label{eqn:2a}
	\end{align}
for some positive $\kappa > 0$.
Then $(T(t))_{t\geq0}$ is an asymptotically stable and contractive $C_0$-semigroup.
\end{proposition}

\begin{proof}
We prove that
	\begin{equation}
	f(x)
	 := \sum_{k=0}^{N-1} \abs{(\H x)^{(k)}(0)}^2,
	 \quad x \in \dom(A_0)
	\end{equation}
has property ASP and use Lemma \ref{lem:properties}.
Let $\beta \in \R$ and $x \in \dom(A_0)$ with
	\begin{equation}
	i \beta x = A_0 x
	 \ \text{and} \
	 f(x) = 0.
	 \nonumber
	\end{equation}
which is a system of ordinary differential equations
	\begin{equation}
	i \beta x(\z)
	 = \sum_{k=0}^N P_k (\H x)^{(k)}(\z),
	 \quad
	 \z \in (0,1)
	\end{equation}
with boundary conditions
	\begin{equation}
	(\H x)^{(k)}(0)
	 = 0,
	 \quad
	 k = 0, 1, \ldots, N-1.
	\end{equation}
Since $P_N$ is invertible the unique solution of this initial value problem is $x = 0$,
so $f$ has property ASP and the result follows from Lemma \ref{lem:3.2}.
\end{proof}

\subsection{First Order Port-Hamiltonian Systems}
\label{subsec:exponential_stability_first_order}

The following exponential stability result can already be found as Theorem III.2 in \cite{VillegasEtAl_2009}.
Here we present a different proof using a frequency domain method.

\begin{proposition}
\label{prop:1st_order}
Let $N = 1$ and $\H \in W_{\infty}^1(0,1;\C^{d \times d})$. 
If the operator $A$ satisfies the assumption
	\begin{equation}
	\Re \sp{A x}{x}_{\H}
	 \leq - \kappa \abs{(\H x)(0)}^2,
	 \qquad x \in D(A)
	 \label{eqn:prop:1st_order-1}
	\end{equation}
for some $\kappa > 0$,
then $A$ generates an exponentially stable and contractive $C_0$-semigroup on the Hilbert space $X$.
\end{proposition}
We remark that in (\ref{eqn:prop:1st_order-1}) we could alternatively choose $- \kappa \abs{(\H x)(1)}^2$ for the right hand side.
For the proof we need the following lemma.
\begin{lemma}
\label{lem:real_part}
Let $Q \in W_{\infty}^1(0,1;\C^{d \times d})$ be a function of self-adjoint operators and $x \in H^1(0,1;\C^d)$. Then
	\begin{equation}
	 \Re \sp{x'}{Qx}_{L_2}
	  = - \frac{1}{2} \sp{x}{Q'x}_{L_2} + \frac{1}{2} \left[x(\z)^* Q(\z) x(\z) \right]_0^1.
	\end{equation}
\end{lemma}
\begin{proof}[Proof of Proposition \ref{prop:1st_order}.]
Theorem \ref{gencontrsgr-thm-7} implies that $A$ generates a contraction $C_0$-semigroup
Let $f: \dom(A_0) \rightarrow \R_+$ be given by $f(x) = \abs{(\H x)(0)}^2$.
We show that $f$ has the ESP property.
By Proposition \ref{prop:asymptotic_stability} property ASP holds and thus we only need to prove the property AIEP.
Let $((x_n,\beta_n))_{n \geq 1} \subset \dom(A_0) \times \R$ be any sequence with $\norm{x_n}_{L_2} \leq c$ and $\abs{\beta_n} \rightarrow \infty$ such that
	\begin{equation}
	A x_n - i \beta_n x_n
	 \xrightarrow{n \rightarrow \infty}
	 0
	 \qquad
	 f(x_n) \xrightarrow{n \rightarrow \infty} 0.
	\end{equation}
Then we obtain the definition of $f$ that
	\begin{equation}
	(\H x_n)(0) \xrightarrow{n \rightarrow \infty} 0.
	\end{equation}
Moreover $\frac{x_n}{\beta_n}$ is bounded in the graph norm $\norm{\cdot}_{A_0}$ and by Lemma \ref{lem:equivalent_norms} we get
	\begin{equation}
	\norm{\frac{(\H x_n)'}{\beta_n}}_{L_2}
	 \leq c,
	 \qquad
	 \text{for all} \ n \in \N.
	\end{equation}
Letting $q \in C^1([0,1];\R)$ with $q(1) = 0$ and having Lemma \ref{lem:real_part} in mind we find
	\begin{align}
	0
	 &\leftarrow \frac{1}{\beta_n} \Re \sp{A x_n - i \beta_n x_n}{i q (\H x_n)'}_{L_2}
	 \nonumber \\
	 &= \frac{1}{\beta_n} \Re \sp{P_1 (\H x_n)'}{i q (\H x_n)'}_{L_2}
	   \nonumber \\
	   & \quad + \frac{1}{\beta_n} \Re \sp{P_0 (\H x_n)}{i q (\H x_n)'}_{L_2}
	   - \Re \sp{x_n}{q (\H x_n)'}_{L_2}
	 \nonumber \\
	 &=\frac{1}{2 \beta_n} \left( \sp{\H x_n}{i q' P_0 (\H x_n)}_{L_2} - \left[ (\H x_n)(\z)^* i q(\z) P_0 (\H x_n)(\z) \right]_0^1 \right)
	   \nonumber \\
	   & \quad - \Re \sp{x_n}{q\H x_n'}_{L_2}
	   - \sp{x_n}{q \H' x_n}_{L_2}
	 \nonumber \\
	 &= \frac{1}{2} \sp{x_n}{(q\H)' x_n}_{L_2}
	   - \frac{1}{2} \left[ x_n(\z)^* q(\z) \H(\z) x_n(\z) \right]_0^1
	   - \sp{x_n}{q \H' x_n}_{L_2}
	   + o(1)
	 \nonumber \\
	 &= - \frac{1}{2} \sp{x_n}{(q \H' - q' \H)x_n}_{L_2}
	   + o(1),
	 \nonumber
	\end{align}
since $(\H x_n)(0) \rightarrow 0, \ q(1) = 0$ and $\abs{\beta_n} \rightarrow \infty$, using integration by parts and $P_1 = P_1^*$.
In particular we may choose $q \leq 0$ such that
	\begin{equation}
	\lambda q - m q' > 0,
	 \qquad
	 \ \z \in [0,1].
	 \nonumber
	\end{equation}
where $\H(\z) \geq m I$ and $\pm \H'(\z) \leq \lambda I$ for a.e. $\z \in [0,1]$,
so $q\H'-q'\H$ is uniformly positive.
This implies
	\begin{equation}
	\norm{x_n}_{L_2}
	 \simeq
	 \norm{x_n}_{L_{2,q\H'-q'\H}}
	 \xrightarrow{n \rightarrow \infty}
	 0.
	\end{equation}
Hence property AIEP holds and exponential stability follows with Lemma \ref{lem:3.2}.
\end{proof}

\subsection{Second Order Port-Hamiltonian Systems}
\label{subsec:exponential_stability_second_order}

As we have seen in the preceding subsection for first order ($N = 1$) port-Hamiltonian systems the sufficient criterion for asymptotic stability in Proposition \ref{prop:asymptotic_stability} even guarantees exponential stability (Proposition \ref{prop:1st_order}).
We now consider second order port-Hamiltonian systems,
i.e.
	\begin{equation}
	A_0 x = P_2 (\H x)'' + P_1 (\H x)' + P_0 (\H x),
	 \quad
	 x \in \dom(A_0) = \H^{-1}H^2(0,1;\C^d)
	 \nonumber
	\end{equation}
and
	\begin{equation}
	A = A_0|_{\dom(A_0)} \ \text{for} \ \dom(A) = \{x \in \dom(A_0): W \left( \begin{array}{c} f_{\partial, \H x} \\ e_{\partial, \H x} \end{array} \right) = 0 \}.
	 \nonumber
	\end{equation}
Adding an additional term $\abs{(\H x)(1)}^2$ (or, $\abs{(\H x)'(1)}^2$) in the dissipativity relation (\ref{eqn:2a}) we again obtain exponential stability.
By means of the example of the one-dimensional Schr\"odinger equation we show that the sufficient criterion for asymptotic stability as in Proposition \ref{prop:asymptotic_stability} is not sufficient for exponential stability in the case $N = 2$.

\begin{proposition}
\label{prop:exp_stability}
Let $N =2$ and $\H \in W^1_{\infty}(0,1;\C^{d \times d})$ and assume
	\begin{equation}
	\Re \sp{Ax}{x}_{\H}
	 \leq - \kappa \left[ \abs{(\H x)(0)}^2 + \abs{(\H x)'(0)}^2 + \left\{ \begin{array}{c} \abs{(\H x)(1)}^2 \\ \text{or} \\ \abs{(\H x)'(1)}^2 \end{array} \right\} \right],
	 \quad
	 x \in \dom(A)
	 \label{eqn:exp_stability_1}
	\end{equation}
for some $\kappa > 0$.
Then $(T(t))_{t\geq0}$ is an exponentially stable and contractive $C_0$-semigroup.
\end{proposition}

Remark that again one may interchange $0$ and $1$ in equation (\ref{eqn:exp_stability_1}).
For the proof, let us first state an auxiliary embedding-and-interpolation result.

\begin{lemma}
\label{lem:interpolation}
Let $0 \leq k < N \in \N_0$ and $\theta \in (0,1)$ such that $\eta := \theta N \in (k + \frac{1}{2},k+1)$.
Then there exist a constant $c_{\theta} > 0$ such that for all $f \in H^N(0,1;\C^d)$
	\begin{equation}
	 \norm{f}_{C^{k}}
	  \leq c_{\theta} \norm{f}_{L_2}^{1-\theta} \norm{f}_{H^N}^{\theta}.
	\end{equation}
Further for $\sigma := \frac{k}{N}$ there exists a constant $c_{\sigma} > 0$ such that for all $f \in H^N(0,1;\C^d)$
	\begin{equation}
	 \norm{f}_{H^k}
	  \leq c_{\sigma} \norm{f}_{L_2}^{1-\sigma} \norm{f}_{H^N}^{\sigma}.
	\end{equation}
\end{lemma}
\begin{proof}
Let $p \in (1, \infty)$ such that $\eta - \frac{1}{2} > k+1 - \frac{1}{p} > k$.
Then by the Sobolev-Morrey Embedding Theorem
	\begin{equation}
	 C^{k}([0,1];\C^d)
	  \hookrightarrow W_p^{k+1}(0,1;\C^d)
	\end{equation}
is continuously embedded.
Further, using the notation of \cite{Triebel_1983}, we have by the theorems of Subsections 3.3.1 and 3.3.6 in \cite{Triebel_1983} that
	\begin{align}
	W_p^N(0,1;\C^d)
	 &=F_{p,2}^{k+1}(0,1;\C^d)
	 \hookrightarrow F_{2,2}^{\eta}(0,1;\C^d)
	 \nonumber \\
	 &= \left( F_{2,2}^0(0,1;\C^d), F_{2,2}^N(0,1;\C^d) \right)_{\theta,2}
	 \nonumber \\
	 &= \left( L_2(0,1;\C^d), H^N(0,1;\C^d) \right)_{\theta,2}
	\end{align}
and the first assertion follows by the interpolation inequality.
The second assertion is a special case of the Gagliardo-Nirenberg inequality.
In the language and with the theory of \cite{Triebel_1983} it results from
	\begin{align}
	 H^k(0,1;\C^d)
	  &= F_{2,2}^k(0,1;\C^d)
	  = \left(F_{2,2}^0(0,1;\C^d), F_{2,2}^N(0,1;\C^d) \right)_{\sigma,2}
	  \nonumber \\
	  &=  \left( L_2(0,1;\C^d), H^N(0,1;\C^d) \right)_{\sigma,2}.
	  \nonumber
	\end{align}
\end{proof}

\begin{proof}[Proof of Proposition \ref{prop:exp_stability}.]
Theorem \ref{gencontrsgr-thm-7} implies that $A$ generates a contraction $C_0$-semigroup.
We show that
 	\begin{equation}
	f(x)
	 := \abs{(\H x)(0)}^2 + \abs{(\H x)'(0)}^2 + \left\{ \begin{array}{c} \abs{(\H x)(1)}^2 \\ \text{or} \\ \abs{(\H x)'(1)}^2 \end{array} \right\},
	 \ x \in \dom(A_0)
	 \label{eqn:prop:exp_stability_2}
	\end{equation}
has the property ESP.
By Proposition \ref{prop:asymptotic_stability} and Lemma \ref{lem:properties} it remains to verify the AIEP property.
Let $(x_n,\beta_n)_{n \geq 1} \subset \dom(A_0) \times \R$ be a sequence with $\norm{x_n}_{L_2} \leq c$ for all $n \in \N$ and $\abs{\beta_n} \rightarrow + \infty$ as $n \rightarrow +\infty$ such that
	\begin{equation}
	i \beta_n x_n - A_0 x_n
	 \xrightarrow{n \xrightarrow{} \infty} 0.
	 \label{eqn:4}
	\end{equation}
By Lemma \ref{lem:equivalent_norms} the sequence $\left(\frac{\H x_n}{\beta_n} \right)_{n\geq1} \subseteq H^2(0,1;\C^d)$ is bounded and by Lemma \ref{lem:interpolation}  $\frac{\H x_n}{\beta_n}$ converges to zero in $C^1([0,1];\C^d)$ (since $\abs{\beta_n} \rightarrow \infty$).
Let $q \in C^2([0,1];\R)$ be some real function.
Integrating by parts and employing the assumptions on the matrices $P_1$ and $P_2$ and Lemma \ref{lem:real_part} we conclude
	\begin{align}
	 0
	  &\longleftarrow \Re \sp{A_0 x_n - i \beta_n x_n}{\frac{iq}{\beta_n} (\H x_n)'}_{L_2}
	  \nonumber \\
	  &= \Re \frac{1}{\beta_n} \sp{P_2 (\H x_n)''}{iq(\H x_n)'}_{L_2}
	   + \frac{1}{\beta_n} \Re \sp{P_1 (\H x_n)'}{iq(\H x_n)'}_{L_2}
	   \nonumber \\ &\quad
	   - \Re \sp{x_n}{q(\H x_n)'}_{L_2}
	   + o(1)
	  \nonumber \\
	  &= - \frac{1}{2 \beta_n} \sp{P_2 (\H x_n)'}{iq'(\H x_n)'}_{L_2}
	   + \frac{1}{2} \sp{x_n}{(q'\H-q\H')x_n}_{L_2}
	   \nonumber \\ &\quad
	   + \frac{1}{2 \beta_n} \left[ (\H x_n)'(\z)^* P_2^* i q(\z) (\H x_n)'(\z) \right]_0^1
	   - \frac{1}{2} \left[ x_n(\z)^* q(\z) \H(\z) x_n(\z) \right]_0^1
	   + o(1)
	  \label{eqn:lem:exp_stability_*}
	\end{align}
and
	\begin{align}
	 0
	  &\longleftarrow \Re \sp{A_0 x_n - i \beta_n x_n}{\frac{iq'}{\beta_n} (\H x_n)}_{L_2}
	  \nonumber \\
	  &= \frac{1}{\beta_n} \Re \sp{P_2 (\H x_n)''}{iq'(\H x_n)}_{L_2}
	   - \sp{x_n}{q'\H x_n}_{L_2}
	   + o(1)
	  \nonumber \\
	  &= - \frac{1}{\beta_n} \sp{P_2 (\H x_n)'}{iq'(\H x_n)'}_{L_2}
	   - \sp{x_n}{q'\H x_n}_{L_2}
	   \nonumber \\ &\quad
	   + \frac{1}{\beta_n} \Re \left[ (\H x_n)'(\z)^* P_2^* i q'(\z) (\H x_n)(\z) \right]_0^1
	   + o(1)
	  \label{eqn:lem:exp_stability_+}
	\end{align}
Subtracting (\ref{eqn:lem:exp_stability_+}) from two times (\ref{eqn:lem:exp_stability_*}) this implies
	\begin{align}
	 0
	  &\longleftarrow \sp{x_n}{(2q'\H - q\H')x_n}_{L_2}
	   + \frac{1}{\beta_n} \left[ (\H x_n)'(\z)^* P_2^* iq(\z) (\H x_n)'(\z) \right]_0^1
	   \nonumber \\ &\quad
	   + \frac{1}{\beta_n} \Re \left[(\H x_n)'(\z)^* P_2^* i q'(\z) (\H x_n)(\z) \right]_0^1
	   - \left[ x_n(\z)^* q(\z) \H(\z) x_n(\z) \right]_0^1.
	\end{align}
Choosing $q \in C^2([0,1];\R)$ such that $q(1) = 0$ and $2q'\H - q\H'$ is uniformly positive this leads in the case that also $f(x_n) \rightarrow 0$ to
	\begin{equation}
	 \norm{x_n}_{L_2}
	  \simeq \norm{x_n}_{q\H'-2q'\H}
	  \xrightarrow{n \rightarrow \infty} 0
	\end{equation}
and thus $f$ also has property AIEP.
\end{proof}

Without proof we remark that using the same proof technique as for Proposition \ref{prop:exp_stability} one obtains the following generalization to port-Hamiltonian systems of even order.

\begin{proposition}
Let $N = 2K \in 2 \N$ be even and $\H \in W^1_{\infty}(0,1;\C^{d \times d})$.
If for some $\kappa > 0$ the dissipativity condition
	\begin{align}
	\Re \sp{Ax}{x}_{\H}
	 &\leq - \kappa f(x)
	 := - \kappa \sum_{\z = 0,1} \sum_{k=0}^{N-1} \alpha_{\z,k} \abs{(\H x)^{(k)}(\z)}^2,
	 \quad
	 x \in \dom(A)
	 \nonumber
	\end{align}
holds true where $\alpha_{\z,k} \geq 0$ are constants such that for some $\z_0 \in \{0,1\}$
	\begin{align}
	 \min (\alpha_{\z_0,0}, \alpha_{\z_0,K})
	  &> 0,
	  \nonumber \\
	 \max (\alpha_{\z_0,k+1}, \alpha_{\z_0, N-k-1})
	  &> 0,
	  \qquad \text{for} \ k = 0, 1, \ldots, K-1,
	  \nonumber \\
	 \max (\alpha_{\z,k}, \alpha_{\z,N-k-1})
	  &> 0,
	  \qquad \text{for} \ \z = 0,1 \ \text{and} \ k = 0, 1, \ldots, K-1,
	  \nonumber
	\end{align}
and $(T(t))_{t\geq0}$ is asymptotically stable,
then $A$ generates an exponentially stable contraction $C_0$-semigroup on $X$.
\end{proposition}

\begin{remark}
One could hope to relax the dissipativity condition in Proposition \ref{prop:exp_stability} to
	\begin{equation}
	\Re \sp{Ax}{x}_{\H} \leq - \kappa \left( \abs{(\H x)(0)}^2 + \abs{(\H x)'(0)}^2 \right),
	 \quad
	 x \in \dom(A).
	\end{equation}
However, the following example shows that even in the case $d = 1$ and $\H \equiv 1$ one generally only has asymptotic (strong) stability.
\end{remark}

\begin{example}[Schr\"odinger Equation]
Let us investigate the one-dimensional Schr\"odinger equation on the unit interval
	\begin{equation}
	i \frac{\partial \omega}{\partial t}(t,\z)
	 + \frac{\partial^2 \omega}{\partial \z^2}(t,\z)
	 = 0,
	 \quad
	 t \geq 0,
	 \
	 \z \in (0,1)
	\end{equation}
with boundary conditions
	\begin{align}
	\frac{\partial \omega}{\partial \z}(t,0)
	 &= - i k \omega(t,0),
	 \nonumber \\
	\frac{\partial \omega}{\partial \z}(t,1)
	 &= \alpha \omega(t,1),
	 \qquad t \geq 0
	\end{align}
for some constants $k > 0$ and $\alpha \in \R \setminus \{0\}$.
The energy functional is given as
	\begin{equation}
	E[\omega(t,\cdot)]
	 = \frac{1}{2} \int_0^1 \abs{\omega(t,\z)}^2 d\z,
	 \qquad t \geq 0
	\end{equation}
and the corresponding port-Hamiltonian operator is
	\begin{align}
	Ax
	 = i x''
	 \quad
	D(A)
	 = \{z \in H^2(0,1): z'(0) = - i k z(0), z'(1) = \alpha z(1) \}. 
	\end{align}
Integrating by parts and using the boundary conditions we deduce
	\begin{align}
	\Re \sp{Ax}{x}_{L_2}
	 &= \Im \left( x'(0)^* x(0) - x'(1)^* x(1) \right)
	 \nonumber \\
	 &= - \frac{1}{2} \left( k \abs{x(0)}^2 + \frac{1}{k} \abs{x'(0)}^2 \right),
	 \quad
	 x \in D(A).
	\end{align}
We claim that the semigroup is \emph{not} exponentially stable, though it is asymptotically (strongly) stable.
For this end we apply Stability Theorem \ref{thm:gearhart} and prove
	\begin{equation}
	\sup_{i \R} \norm{R(\cdot,A)}
	 = \infty.
	 \nonumber
	\end{equation}
Let $\beta > 0$ be arbitrary, hence $i \beta \in \rho(A)$.
For $f \in L_2(0,1)$ we solve $(i \beta - A) x = f$ and obtain the solution
	\begin{align}
	&x(\z)
	 = (R(i \beta,A) f)(\z)
	 \nonumber \\
	 &\ = (\cosh(\sqrt{\beta}\z) - \frac{ik}{\sqrt{\beta}} \sinh(\sqrt{\beta} \z)) x_{\beta,f}(0)
	    + \int_0^{\z} \frac{i}{\sqrt{\beta}} \sinh(\sqrt{\beta}(\z-\xi)) f(\xi) d\xi  
	\end{align}
with the value $x(0) = x_{\beta,f}(0)$ given by
	\begin{equation}
	x_{\beta,f}(0)
	 =\frac{\int_0^1 i (\cosh(\sqrt{\beta}(1-\xi)) - \frac{1}{\sqrt{\beta}} \sinh(\sqrt{\beta}(1-\xi))) f(\xi) d\xi}{(\alpha + ik) \cosh(\sqrt{\beta}) - \left(\frac{i\alpha k}{\sqrt{\beta}} + \sqrt{\beta} \right) \sinh(\sqrt{\beta})}.
	\end{equation}
Now we choose $f = \mathbf{1} \in L_2(0,1)$ and get
	\begin{align}
	&(R(i \beta,A)\mathbf{1})(\z)
	 = (\cosh(\sqrt{\beta} \z) - \frac{ik}{\sqrt{\beta}} \sinh(\sqrt{\beta}\z))
	  \nonumber \\
	  &\ \times \frac{i(\frac{1}{\sqrt{\beta}} \sinh(\sqrt{\beta}) - \frac{1}{\beta} \cosh(\sqrt{\beta}) + \frac{1}{\beta})}{(\alpha + ik) \cosh(\sqrt{\beta}) - \left(\frac{i\alpha k}{\sqrt{\beta}} + \sqrt{\beta} \right) \sinh(\sqrt{\beta})}
	   + \frac{i}{\beta} \cosh(\sqrt{\beta}\z) - \frac{i}{\beta}.
	\nonumber
	\end{align}
Thus for all $\z \in (0,1)$
	\begin{align}
	 &\beta^{3/2} \frac{(R(i \beta, A)\mathbf{1})(\z)}{e^{\sqrt{\beta}\z}}
	  \nonumber \\
	  &\ = i \frac{\cosh(\sqrt{\beta}\z)}{e^{\sqrt{\beta}\z}} \left[ \frac{\beta \sinh(\sqrt{\beta}) - \sqrt{\beta} \cosh(\sqrt{\beta}) + \sqrt{\beta}}{(\alpha + ik) \cosh(\sqrt{\beta}) - \left(\frac{i\alpha k}{\sqrt{\beta}} + \sqrt{\beta} \right) \sinh(\sqrt{\beta})} + \sqrt{\beta} \right]
	   \nonumber \\
	  &\quad + k \frac{\sinh(\sqrt{\beta}\z)}{e^{\sqrt{\beta}\z}} \left[ \frac{\sqrt{\beta} \sinh(\sqrt{\beta}) - \cosh(\sqrt{\beta}) +1}{(\alpha + ik) \cosh(\sqrt{\beta}) - \left(\frac{i\alpha k}{\sqrt{\beta}} + \sqrt{\beta} \right) \sinh(\sqrt{\beta})}\right]
	- \frac{i \sqrt{\beta}}{e^{\sqrt{\beta}\z}}
	   \nonumber \\
	  &\ = k + o(1) + i \frac{\cosh(\sqrt{\beta}\z)}{e^{\sqrt{\beta}\z}}
	   \nonumber \\
	   & \quad \times \frac{- \sqrt{\beta} \cosh(\sqrt{\beta}) + \sqrt{\beta} + (\alpha + ik) \sqrt{\beta} \cosh(\sqrt{\beta}) - i \alpha k \sinh(\sqrt{\beta})}{(\alpha + ik) \cosh(\sqrt{\beta}) - \left(\frac{i\alpha k}{\sqrt{\beta}} + \sqrt{\beta} \right) \sinh(\sqrt{\beta})}
	   \nonumber \\
	  &\ \xrightarrow{\beta \rightarrow \infty} k + i(1 - (\alpha + ik))
	  = 2k + i (1 - \alpha)
	  \not= 0,
	\end{align}
in particular
	\begin{equation}
	\norm{R(i \beta,A)\mathbf{1}}_{L_2} \xrightarrow{\beta \rightarrow + \infty} \infty.
	\nonumber
	\end{equation}
Thus the resolvents cannot be uniformly bounded on the imaginary axis and hence $A$ does not generate an exponentially stable $C_0$-semigroup.
\end{example}

However, for a special class of port-Hamiltonian systems which have some anti-diagonal structure we can weaken the assumptions on the boundary dissipation.

\begin{proposition}
\label{prop:structure_exp_stability}
Let $d$ be even and $0 < \H_1, \H_2 \in W^1_{\infty}(0,1;\C^{d/2 \times d/2})$ and $P_2^* = - P_2 \in \C^{d/2 \times d/2}$ invertible and skew-adjoint, $P_1^* = P_1$ self-adjoint, and $P_0 \in \C^{d \times d}$.
Assume that $A_0$ has the form
	\begin{align}
	A_0 x
	 &= \left( \begin{array}{cc} 0 & P_2 \\ P_2 & 0 \end{array} \right) (\H x)'' + \left( \begin{array}{cc} 0 & P_1 \\ P_1 & 0 \end{array} \right) (\H x)' + P_0 (\H x),
	 \qquad x \in \dom(A_0),
	 \nonumber
	\end{align}
where $\H(\z) = \diag(\H_1(\z), \H_2(\z))$.
Assume there exists some $\kappa > 0$ such that, for all $x = (x_1, x_2) \in \dom(A)$
	\begin{equation}
	\Re \sp{Ax}{x}_{\H}
	 \leq - \kappa \left( \abs{(\H x)(0)}^2 + \left\{ \begin{array}{c} \abs{(\H_1 x_1)'(0)}^2 \\ \text{or} \\ \abs{(\H_2 x_2)'(0)}^2 \end{array} \right\} + \left\{ \begin{array}{c} \abs{(\H_1 x_1)(1)}^2 \\ \text{or} \\ \abs{(\H_1 x_1)'(1)}^2 \end{array} \right\} \right).
	 \nonumber
	\end{equation}
If $(T(t))_{t\geq0}$ is asymptotically stable then it is exponentially stable.
\end{proposition}
Again one may interchange $0$ and $1$ in the dissipativity estimate.
\begin{proof}
The result may be proved in similar fashion as Proposition \ref{prop:exp_stability}.
\end{proof}

\section{Hybrid Systems}
\label{sec:interconnection}

In this section we study stability of hybrid systems. The preconditions for the infinite-dimensional part of the interconnected system stay the same, except for input and output variables which we utilize for interconnection with the finite-dimensional controller.

So, instead of a static boundary condition $W \left( \begin{smallmatrix} f_{\partial, \H x} \\ e_{\partial, \H x} \end{smallmatrix} \right) = 0$ we use (part of) $W \left( \begin{smallmatrix} f_{\partial, \H x} \\ e_{\partial, \H x} \end{smallmatrix} \right)$ to define the input function for the inter\-con\-nection with a finite-di\-men\-sional system and on the other hand use the remaining information from $\left( \begin{smallmatrix} f_{\partial, \H x} \\ e_{\partial, \H x} \end{smallmatrix} \right)$ to define the output map for the interconnection structure.
So let $W, \tilde W \in \C^{Nd \times 2Nd}$ be two full rank matrices and such that the matrix $\left( \begin{smallmatrix} W \\ \tilde W \end{smallmatrix} \right)$ is invertible.
Let $1 \leq m, \tilde m \leq Nd\in \N$ and decompose $W, \tilde W$ as
	\begin{equation}
	W
	 = \left( \begin{array}{c} W_1 \\ W_2 \end{array} \right),
	 \qquad
	\tilde W
	 = \left( \begin{array}{c} \tilde W_1 \\ \tilde W_2 \end{array} \right),
	\nonumber
	\end{equation}
where $W_1 \in \C^{m \times 2nd}$ and $\tilde W_1 \in \C^{\tilde m \times 2Nd}$.
The infinite-dimensional subsystem may then be written as
 	\begin{align}
	\frac{\partial}{\partial t} x(t,\z)
	 &= \sum_{k=0}^N P_k \frac{\partial^k}{\partial \z^k} (\H(\z) x(t,\z)),
	 &t \geq 0, \z \in (0,1),
         \nonumber \\
	u_1(t)
	 &= W_1 \left( \begin{array}{c} f_{\partial,\H x} \\ e_{\partial,\H x} \end{array} \right)(t)
	 =: \mathcal{B}_1 x(t),
	 \nonumber \\
	0
	 = u_2(t)
	 &= W_2 \left( \begin{array}{c} f_{\partial,\H x} \\ e_{\partial,\H x} \end{array} \right)(t)
	 =: \mathcal{B}_2 x(t),
	 \nonumber \\
	y_1(t)
	 &= \tilde W_1 \left( \begin{array}{c} f_{\partial,\H x} \\ e_{\partial,\H x} \end{array} \right)(t)
	 =: \mathcal{C}_1 x(t),
	 \nonumber \\
	y_2(t)
	 &= \tilde W_2 \left( \begin{array}{c} f_{\partial,\H x} \\ e_{\partial,\H x} \end{array} \right)(t)
	 =: \mathcal{C}_2 x(t),	 & t \geq 0.
 	\end{align}
(Further we use the notation $\mathcal{B} := (\mathcal{B}_1, \mathcal{B}_2)$ and $\mathcal{C} = (\mathcal{C}_1, \mathcal{C}_2)$.)
Additionally we consider the space $\Xi = \C^n$ with inner product
	\begin{equation}
	\sp{\xi}{\eta}_{Q_c}
	 := \xi^* Q_c \eta,
	 \qquad
	 \eta, \xi \in \Xi,
	\end{equation}
for some positive $n \times n$-matrix $Q_c = Q_c^* > 0$.
We assume that the finite-dimensional controller has the form
	\begin{align}
	\frac{\partial}{\partial t} \xi(t)
	 &= A_c \xi(t) + B_c u_c(t),
	 \nonumber \\
	y_c(t)
	 &= C_c \xi(t) + D_c u_c(t),
	 \quad
	 t \geq 0
	\end{align}
for some matrices $A_c, B_c, C_c, D_c$ of suitable dimension.
We are interested in situations without external input signal and interconnect the two subsystems by standard feedback interconnection
	\begin{equation}
	u_c = y_1, \qquad u_1 = - y_c.
	\end{equation}
Then we obtain an operator $\A$ on the product space $X \times \Xi$ which we equip with the canonical inner product
	\begin{equation}
	\sp{(x,\xi)}{(y,\eta)}_{\H,Q_c}
	 = \sp{x}{y}_{\H} + \sp{\xi}{\eta}_{Q_c},
	 \quad
	 (x,\xi), (y,\eta) \in X \times \Xi.
	\end{equation}
Namely,
	\begin{equation}
	\A \left( \begin{array}{c} x \\ \xi \end{array} \right)
	 = \left( \begin{array}{cc} A_0 & 0 \\ B_c \mathcal{C}_1 & A_c \end{array} \right) \left( \begin{array}{c} x \\ \xi \end{array} \right)
	\end{equation}
on the domain
	\begin{equation}
	\dom(\A)
	 = \left\{ \left( \begin{array}{c} x \\ \xi \end{array} \right) \in \dom(A_0) \times \Xi: W_{cl} \left( \begin{array}{c} f_{\partial,\H x} \\ e_{\partial,\H x} \\ \xi \end{array} \right) = 0 \right\}
	\end{equation}
with the matrix $W_{cl}$ given by
	\begin{equation}
	W_{cl}
	 = \left( \begin{array}{cc} W_1 + D_c \tilde W_1 & C_c \\ W_2 & 0 \end{array} \right).
	\end{equation}

\subsection{Semigroup Generation}
\label{subsec:interconnection_generation}

Similar to the pure infinite-dimensional case we have the following generation result which includes the case of strictly passive controllers as in Theorem 4 of \cite{RamirezZwartLeGorrec_2013}.

\begin{theorem}
\label{Thm:4.1}
If the operator $\mathcal{A}$ is dissipative, i.e.
	\begin{equation}
	\sp{\mathcal{A}(x,\xi)}{(x,\xi)}_{\H,Q_c}
	 \leq 0,
	 \qquad
	 (x,\xi) \in \dom(\mathcal{A}),
	 \label{eqn:dissipativity_ic}
	\end{equation}
then it generates a contractive $C_0$-semigroup $(\T(t))_{t\geq0}$ on $X \times \Xi$.
Moreover, $\A$ has compact resolvent.
\end{theorem}

\begin{remark}
Similar to the pure infinite-dimensional case one sees that the condition
	\begin{equation}
	\Re P_0
	 := \frac{1}{2} ( P_0 + P_0^*)
	 \leq 0
	 \nonumber
	\end{equation}
is necessary for $\mathcal{A}$ to generate a contraction $C_0$-semigroup.
\end{remark}

For the proof we need the following results which follow from step 2 in the proof of Theorem 4.2 in \cite{LeGorrecZwartMaschke_2005}.

\begin{lemma}
\label{lem:bc-lemma}
Let $d, N \in \N$ and $W \in \C^{Nd \times 2Nd}$ have full rank.
Define $\Phi: H^N(0,1;\C^d) \rightarrow \C^{2Nd} = (\C^d)^{2N}$ by $\Phi_j(x) = x^{(j-1)}(1), \Phi_{j+N}(x) = x^{(j-1)}(0)$ for $j = 1, \ldots, N$.
Then there exists an operator $B \in \B(\C^{Nd};H^N(0,1;\C^d))$ such that
	\begin{equation}
	 (W \circ \Phi)B
	  = I_{\C^{Nd}}.
	  \nonumber
	\end{equation}
\end{lemma}

\begin{corollary}
\label{cor:bc-cor}
Let $W \in \C^{Nd \times 2Nd}$ have full rank and let
	\begin{equation}
	\mathcal{B} x := W \left( \begin{array}{c} f_{\partial, \H x} \\ e_{\partial, \H x} \end{array} \right),
	 \quad x \in \dom(A_0).
	\end{equation}
Then there exists $B \in \B(\C^{Nd};\dom(A_0))$ with $\mathcal{B} B = I_{\C^{Nd}}$.
\end{corollary}

\begin{proof}[Proof of Theorem \ref{Thm:4.1}.]
The operator $\A$ is densely defined.
Namely let $(x,\xi) \in X \times \Xi$ be arbitrary.
Observe that the matrix
	\begin{equation}
	\hat W_{cl}
	 := \left( \begin{array}{c} W_1 + D_c \tilde W_1 \\ W_2 \end{array} \right)
	 = \left( \begin{array}{cc} I_{\C^{Nd}} & \left( \begin{array}{c} D_c \\ 0 \end{array} \right) \end{array} \right) \left( \begin{array}{c} W \\ \tilde W \end{array} \right),
	\end{equation}
has full rank $Nd$ since $\left( \begin{smallmatrix} W \\ \tilde W \end{smallmatrix} \right)$ is invertible and $\left( \begin{smallmatrix} I & \left( \begin{smallmatrix} D_c \\ 0 \end{smallmatrix} \right) \end{smallmatrix} \right)$ has full rank.
Identifying $\Xi \equiv \Xi \times \{0\} \subset \C^{Nd}$, Corollary \ref{cor:bc-cor} shows that there exists $\hat B \in \B(\Xi,\dom(A_0))$ with
	\begin{equation}
	\hat W_{cl} \left( \begin{array}{c} f_{\partial, \H \hat B \xi} \\ e_{\partial, \H \hat B \xi} \end{array} \right)
	 = \left( \begin{array}{c} \xi \\ 0 \end{array} \right),
	 \quad
	 \xi \in \Xi.
	\end{equation}
Moreover since $C_c^{\infty}(0,1;\C^d)$ is dense in $X$ there exists a sequence $\left( \phi_n \right)_{n\geq1} \subset C_c^{\infty}(0,1;\C^d)$ converging to $x + \hat B C_c\xi$.
Note that then 
	\begin{equation}
	\dom(\A) \ni (x_n, \xi_n)
	 := (\phi_n - \hat B C_c \xi, \xi)
	 \xrightarrow{n \rightarrow \infty} (x,\xi) \in X \times \Xi,
	\end{equation}
so $\dom(\A)$ is densely defined.
Thanks to the Lumer-Phillips Theorem II.3.15 in \cite{EngelNagel_2000} and the dissipativity of $\mathcal{A}$,
it remains to check that $\ran (\lambda I - \mathcal{A}) = X \times \Xi$ for some $\lambda > 0$.
To this end let $\lambda > \max(0,s(A_c))$ where $s(A_c) := \sup \{\Re \lambda: \lambda \in \sigma(A_c)\}$ denotes the spectral bound of $A_c$.
Further let $(y, \eta) \in X \times \Xi$ be given.
We are looking for some $(x, \xi) \in \dom(\mathcal{A})$ such that
	\begin{equation}
	\lambda (x, \xi) - \mathcal{A} (x, \xi)
	 = (y, \eta),
	 \nonumber
	\end{equation}
or equivalently
	\begin{align}
	 (\lambda I_X - A_0) x
	  &= y,
	  \nonumber \\
	 (\lambda I_{\Xi} - A_c) \xi - B_c \mathcal{C}_1 x
	  &= \eta,
	  \label{eqn:4.1_1} \\
	 (\mathcal{B}_1 + D_c \mathcal{C}_1) x + C_c \xi
	  &= 0,
	  \nonumber \\
	 W_2 \left( \begin{array}{c} f_{\partial,\H x} \\ e_{\partial,\H x} \end{array} \right) &= 0.
	 \nonumber
	\end{align}
Solving (\ref{eqn:4.1_1}) for $\xi$ and substitution lead to
	\begin{align}
	(\lambda I_X - A_0) x
	 &= y,
	 \nonumber \\
	 \tilde W_{cl} \left( \begin{array}{c} f_{\partial,\H x} \\ e_{\partial,\H x} \end{array} \right)
	  &= \left( \begin{array}{c} -C_c (\lambda I_{\Xi} - A_c)^{-1} \eta \\ 0 \end{array} \right)
	  =: \left( \begin{array}{c} \tilde \eta \\ 0 \end{array} \right)
	\end{align}
where
	\begin{equation}
	 \tilde W_{cl}
	  := \left( \begin{array}{c} W_1 + (D_c + C_c (\lambda I_{\Xi} - A_c)^{-1} B_c) \tilde W_1 \\ W_2 \end{array} \right).
	\end{equation}
Using the operator $\tilde B \in \B(\Xi,\dom(A_0))$ from Corollary \ref{cor:bc-cor} for $\tilde W_{cl}$ we set $x_{new} := x - \tilde B \tilde \eta$ and get the equivalent system
	\begin{align}
	(\lambda I_X - A_0) x_{new}
	 &= y - (\lambda I_X - A_0) \tilde B \tilde \eta,
	 \nonumber \\
	\tilde W_{cl} \left( \begin{array}{c} f_{\partial, \H x_{new}} \\ e_{\partial, \H x_{new}} \end{array} \right)
	 &= 0.
	 \label{eqn:4.1_2}
	\end{align}
Let us consider the operator $\tilde A_{cl} = A_0|_{\dom(\tilde A_{cl})}$ with domain
	\begin{equation}
	\dom(\tilde A_{cl})
	 = \left\{ x \in \dom(A_0): \tilde W_{cl} \left( \begin{array}{c} f_{\partial, \H x} \\ e_{\partial, \H x} \end{array} \right)
	 = 0 \right\}.
	\end{equation}
For any $x \in D(\tilde A_{cl})$ we set $\xi = (\lambda - A_c)^{-1} B_c \mathcal{C}_1 x \in \Xi$ and obtain
	\begin{align}
	W_{cl} \left( \begin{array}{c} f_{\partial, \H x} \\ e_{\partial, \H x} \\ \xi \end{array} \right)
 	 &= \left( \begin{array}{cc} W_1 + D_c \tilde W_1 & C_c \\ W_2 & 0 \end{array} \right) \left( \begin{array}{c} f_{\partial, \H x} \\ e_{\partial, \H x} \\ \xi \end{array} \right)
	 \nonumber \\
 	 &= \tilde W_{cl} \left( \begin{array}{c} f_{\partial, \H x} \\ e_{\partial, \H x} \end{array} \right)
 	 = 0,
	\end{align}
thus $(x,\xi) \in \dom(\mathcal{A})$ and we have
	\begin{align}
	\Re \sp{\tilde A_{cl} x}{x}_{X}
	 &= \Re \sp{\mathcal{A}(x,\xi)}{(x,\xi)}_{X \times \Xi}
	  - \Re \sp{B_c \mathcal{C}_1 x + A_c \xi}{\xi}_{\Xi}
	 \nonumber \\
	 &\leq - \Re \sp{B_c \mathcal{C}_1 x + A_c (\lambda - A_c)^{-1} B_c \mathcal{C}_1 x}{\xi}_{\Xi}
	 \nonumber \\
	 &= - \Re \sp{\lambda (\lambda - A_c)^{-1} B_c \mathcal{C}_1 x}{(\lambda - A_c)^{-1} B_c \mathcal{C}_1 x}_{\Xi}
	 \leq 0,
	\end{align}
for all $x \in \dom(\tilde A_{cl})$.
Hence $\tilde A_{cl}$ generates a contractive $C_0$-semigroup on $X$ by Theorem \ref{gencontrsgr-thm-7}.
Consequently, $(\lambda I - \tilde A_{cl})^{-1} \in \B(X)$ exists and we then get a unique solution $x_{new}$ of (\ref{eqn:4.1_2}) which implies the existence of $(x,\xi) \in \dom(\mathcal{A})$,
	\begin{equation}
	x = x_{new} + \tilde B \tilde \eta,
	 \qquad
	\xi = (\lambda - A_c)^{-1} (\eta + B_c \mathcal{C}_1 x),
	\end{equation}
such that $(\lambda I - \mathcal{A})(x,\xi) = (y,\eta)$.
It follows $\ran (\lambda I - \mathcal{A}) = X$ and the Lumer-Phillips Theorem II.3.15 in \cite{EngelNagel_2000} yields the result.
\end{proof}

\subsection{Asymptotic Behaviour}
\label{subsec:interconnection_stability}

For dissipative hybrid systems we obtain essentially the same stability results as in the pure infinite-dimensional case.

\begin{proposition}
\label{prop:hybrid_strongstability}
Assume that $s(A_c) < 0$ and for a function $f: \dom(A_0) \rightarrow \R_+$
	\begin{equation}
	\Re \sp{\mathcal{A}(x,\xi)}{(x,\xi)}_{\H,Q_c}
	 \leq - f(x),
	 \quad
	 (x,\xi) \in \dom(\mathcal{A}).
	\end{equation}
	\begin{enumerate}
	\item If $f$ has property ASP then $(\T(t))_{t\geq0}$ is asymptotically (strongly) stable.
	\item If $f$ has property AIEP and $\sigma_p(\A) \cap i\R = \emptyset$ then $(\T(t))_{t\geq0}$ is (uniformly) exponentially stable.
	\item If $f$ has property ESP then $(\T(t))_{t\geq0}$ is (uniformly) exponentially stable.
	\end{enumerate}
\end{proposition}

\begin{proof}
1.) Asymptotic stability:
By Theorem \ref{Thm:4.1} $\mathcal{A}$ generates a contractive $C_0$-semigroup and has compact resolvent, so $\sigma(\mathcal{A}) = \sigma_p(\mathcal{A})$.
We want to use Stability Theorem \ref{thm:Arendt-Batty} and thus prove that $i \R \cap \sigma_p(\mathcal{A}) = \emptyset$.
Let $\beta \in \R$ and $(x, \xi) \in \dom(\mathcal{A})$ such that
	\begin{equation}
	i \beta (x, \xi) = \mathcal{A} (x, \xi),
	\end{equation}
so
 	\begin{align}
 	0
	 &= \Re \sp{i \beta(x,\xi)}{(x,\xi)}_{\H,Q_c}
	 \nonumber \\
	 &= \Re \sp{\A(x,\xi)}{(x,\xi)}_{\H ,Q_c}
	 \leq - f(x)
	\end{align}
and by property ASP $x = 0$.
The finite-dimensional component reads
	\begin{equation}
	i \beta \xi = B_c \mathcal{C}_1 x + A_c \xi,
	 \ \text{i.e.} \
	\xi = (i \beta - A_c)^ {-1} B_c \mathcal{C}_1 x,
	\end{equation}
then also $\xi = (i\beta - A_c)^{-1} B_c \mathcal{C}_1 x = 0$.
As a result, $\sigma(\mathcal{A}) = \sigma_p(\mathcal{A}) \subset \C_0^-$ and $(\T(t))_{t\geq0}$ is asymptotically stable due to Stability Theorem \ref{thm:Arendt-Batty}.

2.) Exponential stability:
Let a sequence $\left((x_n, \xi_n, \beta_n)\right)_{n \geq 1} \subset \dom(\mathcal{A}) \times \R$ with $\norm{(x_n, \xi_n)}_{X \times \Xi} \leq c$, $\abs{\beta_n} \xrightarrow{n \rightarrow \infty} + \infty$ such that
	\begin{equation}
	i \beta_n (x_n, \xi_n) - \mathcal{A} (x_n, \xi_n) \xrightarrow{n \rightarrow \infty} 0
	\end{equation}
be given.
Since $\norm{(x_n, \xi_n)}_{X \times \Xi} \leq c$ it especially follows that
	\begin{align}
	f(x)
	 &\leq - \Re \sp{\mathcal{A}(x_n,\xi_n)}{(x_n,\xi_n)}_{\H,Q_c}
	 \nonumber \\
	 &= \Re \sp{i \beta_n (x_n,\xi_n) - \mathcal{A} (x_n,\xi_n)}{(x_n,\xi_n)}_{\H,Q_c}
	 \xrightarrow{n \rightarrow \infty} 0,
	 \nonumber
	\end{align}
so
	\begin{equation}
	f(x_n) \xrightarrow{n \rightarrow \infty} 0.
	\end{equation}
Since
	\begin{equation}
	A_0 x_n - i \beta_n x_n \xrightarrow{n \rightarrow \infty} 0
	\end{equation}
and $f$ has the property AIEP this implies
	\begin{equation}
	x_n \xrightarrow{n \rightarrow \infty} 0.
	\end{equation}
Let us now consider $(\xi_n)_n \subset \Xi$. We have by assumption
	\begin{equation}
	i \beta_n \xi_n - B_c \mathcal{C}_1 x_n - A_c \xi_n \xrightarrow{n \rightarrow \infty} 0,
	 \qquad \text{in} \ \Xi,
	 \nonumber
	\end{equation}
and dividing by $\beta_n \not= 0$ (for $n$ sufficiently large) we get
	\begin{equation}
	\frac{B_c \mathcal{C}_1 x_n}{\beta_n} - i \xi_n \xrightarrow{n \rightarrow \infty} 0,
	 \qquad \text{in} \ \Xi.
	 \nonumber
	\end{equation}
Moreover $\norm{\frac{\A(x_n,\xi_n)}{\beta_n}}$ is bounded and using Lemma \ref{lem:equivalent_norms} we have
	\begin{equation}
	\norm{\cdot}_{\A}
	 \simeq \norm{\cdot}_{\H^{-1}H^N \times \Xi}.
	\end{equation}
Hence $\frac{\H x_n}{\beta_n}$ is a bounded sequence in $H^N(0,1;\C^d)$ and thus by Lemma \ref{lem:interpolation} it is a null sequence in $C^{N-1}([0,1];\C^d)$.
Since $B_c \mathcal{C}_1 x_n$ continuously depends on $\H x_n \in C^{N-1}([0,1];\C^d)$ this implies $\xi_n \rightarrow 0$, so
	\begin{equation}
	(x_n,\xi_n) \xrightarrow{n \rightarrow \infty} 0,
	 \qquad \text{in} \ X \times \Xi.
	 \nonumber
	\end{equation}
From Stability Theorem \ref{thm:gearhart} we deduce exponential stability.

3.\ is a direct consequence of 1.\ and 2.
\end{proof}

\subsection{SIP Controllers with Colocated Input/Output}
\label{subsec:sip_controllers}

We make the following assumption on the infinite-dimensional part.
\begin{assumption}
Assume that the infinite-dimensional port-Hamiltonian system is \emph{passive},
i.e. for all $x \in \dom(A_0)$ it satisfies the balance equation
	\begin{equation}
	\Re \sp{A_0 x}{x}_{\H}
	 \leq \Re \sp{\mathcal{B} x}{\mathcal{C} x}_{\C^{Nd}}.
	 \label{eqn:passive}
	\end{equation}
\end{assumption}
(In particular the corresponding operator on $X$ for $\mathcal{B} x = 0$ generates a contraction $C_0$-semigroup.)
Further we concentrate on finite-dimensional controllers with colocated input and output which are strictly input passive.
\begin{definition}
Let a linear control system
	\begin{align}
	\dot x
	 &= \tilde A x + \tilde B u
	 \nonumber \\
	y
	 &= \tilde C x + \tilde D u
	\end{align}
with state space $\tilde X$ and input and output space $\tilde U = \tilde Y$ (all Hilbert spaces) be given where $\tilde A$ generates a $C_0$-semigroup on $\tilde X$ and $\tilde B \in \B(\tilde U,\tilde X)$, $\tilde C \in \B(\tilde X, \tilde U)$ and $\tilde D \in \B(\tilde U)$ are linear (and continuous) operators.
	\begin{enumerate}
	 \item We say that input and output are \emph{colocated} if $\tilde C \in \B(\tilde X,\tilde U)$ is the adjoint operator of $\tilde B \in \B(\tilde U, \tilde X)$.
	 \item The system is called \emph{strictly input passive} (SIP) if for some $\sigma > 0$ and any solution $x$ one has the estimate
	  \begin{equation}
	 \Re \sp{\dot x}{x}_{\tilde X}
	  \leq \sp{u}{y}_{\tilde U} - \sigma \norm{u}_{\tilde U}^2.
	\end{equation}
	\end{enumerate}
\end{definition}
In our case we assume $m = \tilde m$ and the controller has the form
	\begin{align}
	 \dot \xi
	  &= (J_c - R_c) Q_c \xi + B_c u_c
	  \nonumber \\
	 y_c
	  &= B_c^* Q_c \xi + D_c u_c 
	\end{align}
where $\xi \in \Xi = \C^n$ with inner product $\sp{\xi}{\eta}_{Q_c} = \xi^* Q_c \eta$ for the $n \times n$-matrix $Q_c = Q_c^* > 0$ and $J_c = - J_c^*, \ R_c = R_c^* \geq 0, \ B_c, \ D_c$ matrices of suitable dimension.
For the system to be SIP we demand $D_c = D_c^* \geq \sigma I > 0$.
We then have the following generation result for the operator.
\begin{theorem}
\label{thm:sip_contraction_sgr}
The operator $\A$ generates a contractive $C_0$-semigroup on $X \times \Xi$ and has compact resolvent.
\end{theorem}

\begin{proof}
This is a special case of Theorem \ref{Thm:4.1} since
	\begin{equation}
	\Re \sp{\A (x,\xi)}{(x,\xi)}_{\H,Q_c}
	 \leq \Re \sp{u}{y}_{\C^{Nd}} + \Re \sp{u_c}{y_c}_{\C^m}
	 = 0.
	\end{equation}
\end{proof}

\begin{theorem}
\label{thm:sip_controller_stability}
Let $\sigma(A_c) \subseteq \C_0^-$ and assume that the condition
	\begin{equation}
	\abs{u}^2 + \abs{y_1}^2
	 := \abs{\mathcal{B}x}^2 + \abs{\mathcal{C}_1 x}^2
	 \geq f(x),
	 \quad
	 x \in \dom(A_0)
	\end{equation}
holds where $f: \dom(A_0) \rightarrow \R_+$.
Then
	\begin{enumerate}
	\item If $f$ has property ASP then $(\T(t))_{t\geq0}$ is asymptotically (strongly) stable.
	\item If $f$ has property AIEP and $\sigma_p(\A) \cap i\R = \emptyset$ then $(\T(t))_{t\geq0}$ is (uniformly) exponentially stable.
	\item If $f$ has property ESP then $(\T(t))_{t\geq0}$ is (uniformly) exponentially stable.
	\end{enumerate}
\end{theorem}

\begin{proof}
We already know that $\A$ generates a (contraction) $C_0$-semigroup and has compact resolvent.
Remark that for any $(x, \xi) \in \dom(\A)$ we get
	\begin{align}
	\Re \sp{\A (x,\xi)}{(x, \xi)}_{\H,Q_c}
	 &= \Re \sp{A_0 x}{x}_{\H} + \Re \sp{B_c \mathcal{C}_1 x + (J_c-R_c)Q_c \xi}{\xi}_{Q_c}
	 \nonumber \\
	 &\leq \Re \sp{u}{y}_{\C^m} + \Re \sp{u_c}{B_c^* Q_c \xi}_{\C^m} - \sp{R_c Q_c \xi}{Q_c \xi}_{\C^m}
	 \nonumber \\
	 &\leq \Re \sp{u}{y}_{\C^m} + \Re \sp{u_c}{y_c}_{\C^m} - \sp{u_c}{D_c u_c}_{\C^m}
	 \nonumber \\
	 &\leq - \sigma \abs{u_c}^2.
	 \nonumber
	\end{align}
1.) Assume that $f$ has property ASP.
We prove that $\A$ has no eigenvalues on the imaginary axis.
Let $\beta \in \R$ and $(x,\xi) \in \dom(\A)$ with
	\begin{equation}
	\A (x, \xi)
	 = i \beta (x, \xi)
	 \nonumber
	\end{equation}
be arbitrary.
Then
	\begin{equation}
	0
	 = \Re \sp{i \beta (x,\xi)}{(x,\xi)}_{\H,Q_c}
	 = \Re \sp{\A (x,\xi)}{(x, \xi)}_{\H,Q_c}
	 \leq - \sigma \abs{u_c}^2,
	 \nonumber
	\end{equation}
so $\mathcal{C}_1 x = y_1 = u_c = 0$.
From the equation
	\begin{equation}
	B_c \mathcal{C}_1 x + A_c \xi
	 = i \beta \xi
	 \nonumber
	\end{equation}
and from $\sigma(A_c) \subset \C_0^-$ we then deduce $\xi = 0$ and this also implies $y_c = 0$.
So
	\begin{equation}
	f(x)
	 \leq \abs{u}^2 + \abs{y_1}^2
	 = \abs{u_c}^2 + \abs{y_c}^2
	 = 0
	 \nonumber
	\end{equation}
and hence $f(x) = 0$ and $A_0 x = i \beta x$.
From property ASP we also conclude $x = 0$ and hence $\A$ has no eigenvalues on the imaginary axis.
The result follows from Theorem \ref{thm:Arendt-Batty}.

2.) Let us assume $(\T(t))_{t\geq0}$ is asymptotically stable and $f$ has property AIEP.
Let $((x_n,\xi_n),\beta_n)_{n \geq 1} \subset \dom(\A) \times \Xi$ be any sequence with $\norm{x_n} \leq c, \ \abs{\beta_n} \rightarrow \infty$ and
	\begin{equation}
	\A (x_n, \xi_n) - i \beta_n (x_n, \xi_n)
	 \xrightarrow{n \rightarrow \infty} 0
	 \quad
	 \text{in} \ X \times \Xi.
	\end{equation}
We then especially have
	\begin{equation}
	0
	 \leftarrow \sp {(\A - i \beta_n) (x_n, \xi_n)}{(x_n,\xi_n)}_{\H,Q_c}
	 \leq - \sigma \abs{u_{c,n}}^2,
	 \nonumber
	\end{equation}
thus $y_{1,n} = \mathcal{C}_1 x_n = u_{c,n} \rightarrow 0$.
Also
	\begin{equation}
	B_c u_{c,n} + A_c \xi_n - i \beta_n \xi_n
	 =: \eta_n
	 \xrightarrow{n \rightarrow \infty} 0
	 \nonumber
	\end{equation}
and since $\sup_{i \R} \norm{R(\cdot,A_c)} < + \infty$ we obtain
	\begin{equation}
	\xi_n
	 = R(i \beta_n, A_c) (B_c u_{c,n} - \eta_n)
	 \xrightarrow{n \rightarrow \infty} 0
	\end{equation}
and then also
	\begin{equation}
	- u_{1,n}
	 = y_{c,n}
	 = B_c^* Q_c \xi_n + D_c u_{c,n}
 	 \xrightarrow{n \rightarrow \infty} 0.
	 \nonumber
	\end{equation}
It follows
	\begin{equation}
	 0
	  \leq f(x_n)
	  \leq \abs{\mathcal{B} x_n}^2 + \abs{\mathcal{C}_1 x_n}^2
	  \xrightarrow{n \rightarrow \infty} 0
	\end{equation}
and since $A_0 x_n - i \beta_n x_n \rightarrow 0$ we obtain $x_n \rightarrow 0$ from the AIEP property,
i.e. $(x_n,\xi_n) \rightarrow 0$ and the assertion follows from Theorem \ref{thm:gearhart}.

3.\ is a direct consequence of 1.\ and 2.
\end{proof}

As a result, Theorem 14 of \cite{RamirezZwartLeGorrec_2013} follows directly from Proposition \ref{prop:1st_order} and Theorem \ref{thm:sip_controller_stability}.

\begin{corollary}
Let $N=1, \ \sigma(A_c) \subseteq \C_0^-$ and assume that for some $\kappa > 0$
	\begin{equation}
	\abs{u}^2 + \abs{y_1}^2
	 := \abs{\mathcal{B}x}^2 + \abs{\mathcal{C}_1 x}^2
	 \geq \kappa \abs{(\H x)(0)}^2,
	 \quad
	 x \in \dom(A_0).
	\end{equation}
Then the controller exponentially stabilizes the port-Hamiltonian system,
i.e. the semigroup $(\T(t))_{t\geq0}$ is exponentially stable.
\end{corollary}

For the case $N = 2$ the following follows directly from the results of Subsection \ref{subsec:exponential_stability_second_order}.

\begin{corollary}
\label{cor:sip_second_order}
Let $N = 2, \ \sigma(A_c) \subseteq \C_0^-$.
 \begin{enumerate}
  \item If for some $\kappa > 0$ and all $x \in \dom(A_0)$
	\begin{equation}
	 \abs{\mathcal{B} x}^2 + \abs{\mathcal{C}_1 x}^2
	  \geq \kappa \left( \abs{\H x (0)}^2 + \abs{(\H x)'(0)}^2 \right),
	\end{equation}
	then the semigroup $(\T(t))_{t\geq0}$ is asymptotically stable.
  \item If  for some $\kappa > 0$ and all $x \in \dom(A_0)$
	\begin{equation}
	 \abs{\mathcal{B} x}^2 + \abs{\mathcal{C}_1 x}^2
	  \geq \kappa \left( \abs{(\H x)(0)}^2 + \abs{(\H x)'(0)}^2 + \left\{ \begin{array}{c} \abs{(\H x)(1)}^2 \\ \text{or} \\ \abs{(\H x)'(1)}^2 \end{array} \right\} \right),
	\end{equation}
	then $(\T(t))_{t\geq0}$ is exponentially stable.
  \item If $A_0$ has the structure as in Proposition \ref{prop:structure_exp_stability} and $(\T(t))_{t\geq0}$ is asymptotically stable and for some $\kappa > 0$ and all $x \in \dom(A_0)$
	\begin{equation}
	 \abs{\mathcal{B} x}^2 + \abs{\mathcal{C}_1 x}^2
	  \geq \kappa \left( \abs{(\H x)(0)}^2 + \left\{ \begin{array}{c} \abs{(\H_1 x_1)'(0)}^2 \\ \text{or} \\ \abs{(\H_2 x_2)'(0)}^2 \end{array} \right\} + \left\{ \begin{array}{c} \abs{(\H_1 x_1)(1)}^2 \\ \text{or} \\ \abs{(\H_1 x_1)'(1)}^2 \end{array} \right\} \right),
	\end{equation}
	then $(\T(t))_{t\geq0}$ is exponentially stable.
 \end{enumerate}
\end{corollary}

\section{Example: Stabilization of the Euler-Bernoulli beam}
\label{sec:examples}

We illustrate our theoretical results by means of the Euler-Bernoulli beam equation
	\begin{equation}
	 \rho(\z) \omega_{tt}(t,\z) + \frac{\partial^2}{\partial \z^2}(EI(\z) \omega_{\z \z}(t,\z))
	  = 0,
	  \quad
	  \z \in (0,1), \ t \geq 0.
	\end{equation}
The distributed parameters $\rho, EI$ are assumed to be of class $W_{\infty}^1(0,1;\R)$ and strictly positive.
The energy of any sufficient smooth solution is given by
	\begin{equation}
	E(t)
	 = \frac{1}{2} \int_0^1 \rho(\z) \abs{\omega_t(t,\z)}^2 + EI(\z) \abs{\omega_{\z\z}(t,\z)}^2 d\z.
	\end{equation}
For the standard port-Hamiltonian formulation we thus set
	\begin{equation}
	 x_1
	  := \rho \omega_t,
	  \quad
	 x_2
	  := \omega_{\z\z}
	  \nonumber
	\end{equation}
and denote by $\H = \diag (\H_1, \H_2) = \diag(\frac{1}{\rho}, EI)$ the Hamiltonian density matrix function.
The Euler-Bernoulli beam equation may then be written as
	\begin{equation}
	 \frac{\partial}{\partial t} x(t,\z)
	  = \left( \begin{array}{cc} 0 & -1 \\ 1 & 0 \end{array} \right) (\H x)''(t,\z).
	\end{equation}
For the related operator $A_0$ we obtain the balance equation
	\begin{equation}
	 \Re \sp{A_0 x}{x}_{\H}
	  = \Re \left[ (\H_1 x_1)'(\z)^* (\H_2 x_2)(\z) - (\H_1 x_1)(\z)^* (\H_2 x_2)'(\z) \right]_0^1.
	  \label{eqn:euler-bernoulli_balance}
	\end{equation}
\begin{example}[Clamped Left End]
\label{exa:euler-bernoulli_cle}
In \cite{ChenEtAl_1987} the authors consider clamped left end boundary conditions and static feedback at the right end.
	\begin{align}
	\omega(t,0)
	 = \omega_{\z}(t,0)
	 &= 0
	 &(\text{clamped left end})
	 \nonumber \\
	\frac{\partial}{\partial \z}(EI \omega_{\z \z})(t,1)
	 &= \alpha_1 \omega_t(t,1)
	 \nonumber \\
	(EI \omega_{\z \z})(t,1)
	 &= - \alpha_2 \omega_{t\z}(t,1)
	 &(\text{static feedback})
	\end{align}
for feedback constants $\alpha_1, \alpha_2 \geq 0$.
The corresponding operator $A_{\alpha_1, \alpha_2}$ on $X$ is then given by
	\begin{align}
	A_{\alpha_1,\alpha_2} x
	 &= A_0|_{\dom(A_{\alpha_1,\alpha_2})}
	 \nonumber \\
	\dom(A_{\alpha_1, \alpha_2})
	 &= \{ x \in \dom(A_0): \ (\H_1 x_1)(0) = (\H_1 x_1)'(0) = 0,
	 \nonumber \\
	 & \quad (\H_2 x_2)'(1) = \alpha_1 (\H_1 x_1)(1), \ (\H_2 x_2)(1) = - \alpha_2 (\H_1 x_1)'(1) \}.
	 \nonumber
	\end{align}
Clearly if $\alpha_1 = \alpha_2 = 0$ then $\Re \sp{A_{0,0} x}{x}_{\H} = 0$ and $A_{0,0}$ generates an unitary $C_0$-group.
Next, let $\alpha_1, \alpha_2 > 0$.
Then we obtain from equation (\ref{eqn:euler-bernoulli_balance}) for $x \in \dom(A_{\alpha_1,\alpha_2})$ that
	\begin{align}
	\Re \sp{A_{\alpha_1,\alpha_2}x}{x}_{\H}
	 &\leq - \kappa \left( \abs{(\H x)(1)}^2 + \abs{(\H x)'(1)}^2  + \abs{(\H_1 x_1)(0)}^2 + \abs{(\H_1 x_1)'(0)}^2\right)
	 \nonumber
	\end{align}
Asymptotic and then exponential stability follow by Propositions \ref{prop:asymptotic_stability} and \ref{prop:structure_exp_stability}.
Last we investigate the case $\alpha_1 > 0$ and $\alpha_2 = 0$.
From the boundary conditions we obtain
	\begin{align}
	&\Re \sp{A_{\alpha_1,\alpha_2}x}{x}_{\H}
	 \nonumber \\
	 & \quad \leq - \kappa \left( \abs{(\H x)(1)}^2 + \abs{(\H_1 x_1)(0)}^2 + \abs{(\H_1 x_1)'(0)}^2 + \abs{(\H_2 x_2)'(1)}^2  \right)
	 \label{eqn:exa-eulerbernoulli-1}
	\end{align}
for all $x \in \dom(A_{\alpha_1,\alpha_2})$ and some $\kappa > 0$.
By ODE techniques one finds $\sigma_p(A_{\alpha_1,\alpha_2}) \cap i \R = \emptyset$ after which exponential stability again follows from Proposition \ref{prop:structure_exp_stability}, a result first proved in \cite{ChenEtAl_1987}.
\end{example}

\begin{remark}
Unfortunately our theoretical results do not cover the case $\alpha_1 = 0$ and $\alpha_2 > 0$. Although we may prove asymptotic stability in a similar way as before, for the corresponding operator $A_{0,\alpha_2}$ we only have the estimate
	\begin{equation}
	\Re \sp{A_{0,\alpha_2}x}{x}_{\H}
	 \leq - \kappa ( \abs{\H_1 x_1(0)}^2 + \abs{(\H_1 x_1)'(0)}^2 + \abs{(\H_2 x_2)(1)}^2 + \abs{(\H_1 x_1)'(1)}^2)
	 \nonumber
	\end{equation}
The uniform energy decay for this case has been proved in \cite{ChenEtAl_1987a}.
\end{remark}

The second example shows how interconnection structures naturally appear in the modelling of beams with a mass at one end.

\begin{example}[Both Ends Free]
Assume that both ends are free and there is a mass at the tip (as in \cite{GuoHuang_2004}).
	\begin{align}
	 EI(\z) \frac{\partial^2}{\partial \z^2} \omega(\z,t)|_{\z = 1}
	  &= - \frac{\partial}{\partial \z} \left( EI(\z) \frac{\partial}{\partial \z^2} \omega(t,\z) \right)|_{\z = 1}
	  = 0
	  \nonumber \\
	 \left. EI(\z) \frac{\partial^2}{\partial \z^2} \omega(t,\z) \right|_{\z = 0}
	  &= k_1 \frac{\partial}{\partial \z} \omega(t,0) + k_2 \frac{\partial^2}{\partial t \partial \z} \omega(t,0)
	  \nonumber \\
	 - \left. \frac{\partial}{\partial \z} \left( EI(\z) \frac{\partial^2}{\partial \z^2} \omega(t,\z) \right) \right|_{\z = 0}
	  &= k_3 \omega(t,0) + k_4 \frac{\partial}{\partial t} \omega(t,0)
	  \nonumber
	\end{align}
for feedback constants $k_j > 0$.
The system's total energy is
	\begin{align}
	\tilde E(t)
	 &= \frac{1}{2} \int_0^1 \left( EI(\z) \abs{\frac{\partial^2}{\partial \z^2} \omega(t,\z)}^2 + \rho(\z) \abs{\frac{\partial}{\partial t} \omega(t,\z)}^2 \right) d\z
	  \nonumber \\
	  &\  +\frac{1}{2} \left( k_1 \abs{\frac{\partial}{\partial \z} \omega(t,0)}^2 + k_3 \abs{\omega(t,0)}^2 \right),
	 \nonumber
	\end{align}
i.e. it decomposes into a continuous and a discrete part.
To interpret this as a hybrid system we choose $\xi = (\omega_{\z}(0), \omega(0))$ and the input- and output functions
	\begin{align}
	u_1(t)
	 &= \left( \begin{array}{c} (\H_1 x_1)'(t,0) \\ (\H_1 x_1)'(t,0)\end{array} \right),
	 \nonumber \\
	0
	 = u_2(t)
	 &= \left( \begin{array}{c} (\H_2 x_2)(t,1) \\ (\H_2 x_2)'(t,1) \end{array} \right)
	 \nonumber \\
	y(t)
	 &= \left( \begin{array}{c} - (\H_2 x_2)'(t,0) \\ (\H_2 x_2)(t,0) \end{array} \right).
	\end{align}
The infinite-dimensional part is passive due to equation (\ref{eqn:euler-bernoulli_balance}).
By feedback interconnection $u_1 = - y_c$ and $u_c = y$ with the finite-dimensional system
	\begin{align}
	\frac{\partial}{\partial t} \xi(t)
	 &= - \left( \begin{array}{cc} \frac{1}{k_2} & 0 \\ 0 & \frac{1}{k_4} \end{array} \right) \left( \begin{array}{cc} k_1 & 0 \\ 0 & k_3 \end{array} \right) \xi(t)
	   + \left( \begin{array}{cc} \frac{1}{k_2} & 0 \\ 0 & \frac{1}{k_4} \end{array} \right) u_c(t)
	 \nonumber \\
	y_c(t)
	 &= \left( \begin{array}{cc} \frac{1}{k_2} & 0 \\ 0 & \frac{1}{k_4} \end{array} \right) \left( \begin{array}{cc} k_1 & 0 \\ 0 & k_2 \end{array} \right) \xi(t)
	   + \left( \begin{array}{cc} k_1 & 0 \\ 0 & k_3 \end{array} \right) u_c(t),
	\end{align}
setting $Q_c = D_c = \left( \begin{smallmatrix} k_1 & 0 \\ 0 & k_3 \end{smallmatrix} \right)$ and $R_c = B_c = \left( \begin{smallmatrix} \frac{1}{k_2} & 0 \\ 0 & \frac{1}{k_4} \end{smallmatrix} \right)$ we see that this is indeed a port-Hamiltonian system interconnected in a energy preserving way with an exponentially stable SIP controller with colocated input and output.
Since
	\begin{equation}
	\abs{\mathcal{B} x}^2 + \abs{\mathcal{C}_1 x}^2
	 = \abs{(\H x)(0)}^2 + \abs{(\H x)'(0)}^2 + \abs{(\H_2 x_2)(1)}^2 + \abs{(\H_2 x_2)'(1)}^2
	 \nonumber
	\end{equation}
uniform exponential energy decay for the hybrid system follows from Corollary \ref{cor:sip_second_order}.
\end{example}

\section*{Acknowledgements}
We would like to thank Hafida Laasri for careful proof-reading and giving valuable suggestions.

\end{document}